\documentclass[hidelinks,onefignum,onetabnum]{siamart220329}

\newcommand{\boldv}{\boldsymbol{v}}
\newcommand{\boldbeta}{\boldsymbol{\beta}}
\newcommand{\boldx}{\boldsymbol{x}}
\newcommand{\bolde}{\boldsymbol{e}}
\newcommand{\boldb}{\boldsymbol{b}}
\newcommand{\boldd}{\boldsymbol{d}}
\newcommand{\boldz}{\boldsymbol{z}}
\newcommand{\boldy}{\boldsymbol{y}}
\newcommand{\bolds}{\boldsymbol{s}}
\newcommand{\boldr}{\boldsymbol{r}}
\newcommand{\boldu}{\boldsymbol{u}}
\newcommand{\boldi}{\boldsymbol{i}}
\newcommand{\boldzero}{\boldsymbol{0}}

\newcommand{\boldxi}{\boldsymbol{\xi}}
\newcommand{\boldnu}{\boldsymbol{\nu}}
\newcommand{\Mperp}{M_{\perp}}
\newcommand{\coefB}{\tfrac{1}{\sqrt{2}}}
\newcommand{\phm}{\phantom{-}}

\usepackage{mathdots} 
\usepackage{pgfplots}
\pgfplotsset{compat=1.18}
\usepackage{pgfplotstable}
\usepackage{todonotes}

\usepackage{lipsum}
\usepackage{amsfonts}
\usepackage{graphicx}
\usepackage{epstopdf}
\usepackage{algorithmic}
\usepackage{textcomp}
\usepackage{xcolor}
\usepackage[nolist,nohyperlinks]{acronym}
\usepackage{tikz}
\usetikzlibrary{matrix, positioning, 3d, patterns}
\usepackage{multirow}
\usepackage{array}
\usepackage{booktabs}
\usepackage[section]{placeins}
\usepackage{siunitx}
\usepackage{balance}
\usepackage{comment}
\usepackage{amsmath}
\usepackage{amssymb}
\usepackage{enumitem}

\usepackage{xspace}
\usepackage{nameref}
\usepackage[capitalize]{cleveref}
\usepackage{hyperref}

\DeclareMathOperator{\diag}{diag}

\theoremstyle{plain}

\newtheorem{assumption}[theorem]{Assumption}
\newtheorem{remark}{Remark}

\usepackage{macros}

\ifpdf
  \DeclareGraphicsExtensions{.eps,.pdf,.png,.jpg}
\else
  \DeclareGraphicsExtensions{.eps}
\fi

\Crefname{subsection}{Section}{Sections}
\crefname{subsection}{section}{sections}

\crefname{assumption}{assumption}{assumptions}
\Crefname{assumption}{Assumption}{Assumptions}

\headers{3D-DFT using compressed sensing}{Kuang et al.}

\title{
Recovering sparse DFT from missing Signals via interior point method on GPU
\thanks{
This work was supported by the U.S. Department of Energy, Office of Science, Advanced Scientific Computing Research Program under contracts DE-AC02-06CH11357.}}

\author{Wei Kuang\thanks{Department of Statistics, University of Chicago.}
\and Vishwas Rao\thanks{Mathematics and Computer Science Division, Argonne National Laboratory.}
\and Alexis Montoison\thanks{Mathematics and Computer Science Division, Argonne National Laboratory.}
\and François Pacaud \thanks{Centre Automatique et Systèmes, Mines Paris-PSL.}
\and Mihai Anitescu \thanks{Mathematics and Computer Science Division, Argonne National Laboratory.}
}

\makeatletter
\newcommand*{\addFileDependency}[1]{
  \typeout{(#1)}
  \@addtofilelist{#1}
  \IfFileExists{#1}{}{\typeout{No file #1.}}
}
\makeatother

\usepackage{lmodern}

\begin{document}

\maketitle
\tableofcontents

\begin{abstract}
We propose a method to recover the sparse discrete Fourier transform (DFT) of a signal that is both noisy and potentially incomplete, with missing values.
The problem is formulated as a penalized least-squares minimization based on the inverse discrete Fourier transform (IDFT) with an $\ell_1$-penalty term, reformulated to be solvable using a primal-dual interior point method (IPM).
Although Krylov methods are not typically used to solve Karush-Kuhn-Tucker (KKT) systems arising in IPMs due to their ill-conditioning, we employ a tailored preconditioner and establish new asymptotic bounds on the condition number of preconditioned KKT matrices.
Thanks to this dedicated preconditioner --- and the fact that FFT and IFFT operate as linear operators without requiring explicit matrix materialization --- KKT systems can be solved efficiently at large scales in a matrix-free manner.
Numerical results from a Julia implementation leveraging GPU-accelerated interior point methods, Krylov methods, and FFT toolkits demonstrate the scalability of our approach on problems with hundreds of millions of variables, inclusive of real data obtained from the diffuse scattering from a slightly disordered Molybdenum Vanadium Dioxide crystal.
\end{abstract}

\section{Introduction}
\label{sec:introduction}

The discrete Fourier transform (DFT) is a fundamental tool for converting signals from the time domain to the frequency domain, revealing their spectral content.
Its applications span audio and video processing, GPS, and medical imaging.
The fast Fourier transform (FFT) algorithm efficiently computes the DFT, enabling its use in large-scale problems.
When a signal’s energy is concentrated in a few frequencies, the resulting DFT has a sparse structure that can be exploited:
The sparse Fast Fourier transform (sFFT), introduced by \cite{Hassanieh2012Simple}, computes the sparse DFT more efficiently than FFT. Noisy signals present additional challenges, as energy leakage can occur in the decomposition.
When the signal-to-noise ratio is sufficient, most energy remains concentrated in a few dominant frequencies.
The sFFT can approximate sparse DFTs under such conditions, with theoretical guarantees for error bounds \cite{Hassanieh2012Nearly}.

This work extends the sFFT framework to handle the more complex setting of noisy signals with missing values, a problem that is particularly relevant in X-ray image processing for crystals.
A key motivation for this work lies in materials science, particularly in the study of order-disorder transitions in crystals.
In that case, the useful signal, the "disorder", is small and masked by the rest of the structure which is still primarily crystalline. Therefore the X-ray diffraction from such a structure will not have useful signal around the Bragg peaks of the unperturbed structure and needs to be removed before attempting to reconstruct the features of the disorder.
The missing signal values at those locations implies that the inversion of the disorder is underdetermined.
These investigations often focus on analyzing the local structure of crystals using the total pair distribution function (PDF), which is the Fourier transform of the total scattering intensity \cite{Takeshi2012Structure}.
For diffraction experiments on single crystals, this is referred to as the total three-dimensional PDF (3D-PDF) \cite{Weber2012three, Welberry1994Interpretation}, providing detailed information about interatomic vectors.

While single crystal diffraction is more experimentally challenging than powder diffraction, it yields three-dimensional diffraction data.
Unlike the powder PDF, which describes the distribution of distances between atom pairs, the 3D-PDF provides full information about interatomic vectors.
The 3D-$\Delta$PDF isolates disorder-related features by removing Bragg peaks that represent the average structure, leaving only information about structural deviations \cite{Weber2012three}.
Recovering the 3D-$\Delta$PDF from total scattering data is inherently challenging, as Bragg peaks must be accurately subtracted and the remaining data reconstructed with minimal artifacts.
Earlier methods, such as the punch-and-fill algorithm \cite{kobas2005structural}, including our own innovation such as the Julia package \texttt{LaplaceInterpolation.jl} \cite{RH2022},
relies on interpolation to address this problem.
However, these approaches often introduced sharp edges, leading to unintended ripples or noise in the resulting PDF (see \S\ref{sec:real_datasets} for more details).

The 3D-$\Delta$PDF is inherently sparse, with nonzero values appearing only at discrete features encoding pairwise interatomic vector probabilities.
We formulate the signal recovery as a regularized least-squares problem with an $\ell_1$-norm penalty to enforce sparsity.
Our approach reduces interpolation artifacts such as ripples and yields a reconstruction that is more consistent with physical interpretations.
It leverages DFT sparsity by employing an $\ell_1$ penalty, a computationally feasible relaxation of the $\ell_0$ constraint.
It promotes sparsity while minimizing the squared error between the observed and predicted signals.
The reconstruction process uses only available data plus its sparsity structure, ensuring more accurate signal recovery compared to the punch-and-fill method which needs to impute the signal through interpolation in a large image area.
Since the diffraction image is three dimensional \cite{Welberry1994Interpretation} the problem is much larger than compressed sensing over the much more common two dimensional images and routinely results in several hundred million voxels. Moreover, continuous improvements in sensor technology (for example, the advanced photon source (APS) upgrade) will increase the number of voxels significantly \cite{streiffer2015early}. 
To simplify computations, the problem is reformulated in the real domain through one-to-one mappings, which eliminate the need to explicitly manage DFT conjugate symmetry constraints.
This transformation reformulates the problem as a LASSO problem, a well-established framework for sparse regression.
It also enables the use of real FFT and IFFT, which are computationally more efficient in terms of storage and computation.

There exist numerous first-order methods to solve the compressed sensing problem: Examples include gradient project sparse reconstruction problem \cite{Figueiredo_2007}, proximal forward-backward splitting method~\cite{combettes2005signal}, Nesterov’s algorithm \cite{Becker_2011}, fixed point continuation active set \cite{Wen_2010}, and alternating direction method of multipliers (ADMM) \cite{boyd2011distributed}. While first-order methods are widely used due to their simplicity and scalability, the second-order
methods benefit from stronger convergence guarantees \cite{candes2005l1} and may be very competitive
with first-order methods provided the required linear algebra can run efficiently on the target architectures \cite{Fountoulakis_2013}.
For that reason,  we employ a primal-dual IPM \cite{wright1997primal},
in line with the methods proposed earlier in \cite{Kim2007Interior,saunders2002pdco}.
IPM requires solving a sequence of Newton linear systems, which can be computationally expensive.
In particular, the Newton linear systems become ill-conditioned close to convergence, making it increasingly challenging to solve the linear systems using iterative methods. As a consequence, the linear systems in IPMs are solved primarly with direct linear algebra \cite{wright2001effects,gondzio2012interior}.
Several iterative approaches have been proposed~\cite{Gondzio_2010,gondzio2012interior}, but all require an efficient preconditioner to be tractable.
Fortunately, provided such a preconditioner exists, the matrix-free IPM variant has a structure favorable for GPU acceleration~\cite{Smith_2012}.
Here, the large-scale dimension of the compressed sensing problem anyway prohibits the use of direct solvers in IPM, leaving
the matrix-free variant as the only tractable option.
The authors in \cite{saunders2002pdco} are using a generic IPM method to solve
compressed sensing problems, using \textsc{LSQR} \cite{paige-saunders-1982} to solve the Newton systems.
In a direction very similar to our work, \cite{Fountoulakis_2013} also proposed an IPM where the Newton systems are solved using preconditioned \textsc{CG}: by
using the restricted isometry property (RIP), they proved the
eigenvalues of the preconditioned Newton matrix are clustered around one, allowing \textsc{CG} to converge in few iterations. For a different regularization than the one discussed here,
we note that \cite{ghannad-orban-saunders-2022} proposed a version of an interior-point method whose Newton system matrix has a bounded condition number and can be in principled be approached with preconditioned Krylov methods as we propose here as well.

Several packages exist for solving interior-point method (IPM) problems, but few offer support for matrix-free linear solvers.
This limitation applies to libraries such as \texttt{IPOPT} \cite{W_chter_2005} and \texttt{KNITRO} \cite{byrd2006k}.
In contrast, the package \texttt{MadNLP.jl} \cite{shin2023accelerating} supports various KKT formulations and matrix-free linear solvers, enabling users to choose solvers tailored to their problem structure.

Given the computational demands of high-resolution 3D-$\Delta$PDF recovery, we accelerate the process using GPUs and provide a Julia implementation. Our code relies on primal-dual IPM (\texttt{MadNLP.jl}), Krylov solvers (\texttt{Krylov.jl}), and FFT toolkit (\texttt{cuFFT}): all packages are optimized to leverage GPU parallelism, ensuring scalability for problems involving hundreds of millions of variables. We outline the contributions of the paper below.

\section*{Contributions}\label{sec:contributions}
This paper introduces a scalable method for sparse discrete Fourier transform (sDFT) recovery under noisy and incomplete conditions, with the following key contributions:
\begin{itemize}
    \item \textbf{Sparse DFT formulation}: Reformulation of the recovery problem as an $\ell_1$-regularized least-squares optimization in the real domain where it is sparse.
    \item \textbf{New analysis framework:} We introduce a new framework in \S \ref{sec:formulation} for analyzing the preconditioned CG steps for a class of primal-dual IPMs and bounding the condition number of the iteration matrix on one of the standard IPM central-path neighborhoods. When applied to the problem at hand it produces estimates comparable to \cite{Fountoulakis_2013} for the same problem class. However, as we will explain in \S \ref{sec:formulation}, our results rely only on classical estimates for interior point methods \cite{wright2001effects} which makes them applicable in principle beyond the cases discussed here, including to $\ell_1$-regularized quadratic programming. We discuss this issue in detail in Remark \ref{r:jacek}.
    \item \textbf{Matrix-free KKT operators}: Development of matrix-free KKT operators for each iteration of the interior-point method (IPM).
    The approach leverages tailored preconditioners, Krylov subspace methods, and FFT-based techniques to efficiently solve large-scale systems.
    \item \textbf{GPU-accelerated implementation}: Building on our previous work with Julia-based GPU-accelerated libraries such as \texttt{MadNLP.jl} and \texttt{Krylov.jl}, we have developed \texttt{CompressedSensingIPM.jl}.
    This package provides an interface specifically designed to simplify the resolution of compressed sensing problems, while leveraging the capabilities of these high-performance libraries. We will demonstrate this package on problems with hundreds of millions of variables in \S\ref{sec:num-results}.
    \item \textbf{Applications in crystallography}: The approach helps understand the order-disorder transitions in materials by mitigating artifacts from traditional interpolation methods in X-ray crystallography. We demonstrate this for accurate recovery of 3D-$\Delta$PDFs on real X-ray crystallographical data from a slightly disordered Molybdenum Vanadium Dioxide crystal in \S \ref{sec:applications}.
\end{itemize}

\section*{Summary}\label{sec:summary}
The remainder of this paper is organized as follows.
In \S\ref{sec:notations}, we introduce the notations.
\S\ref{sec:formulation} details the formulation of the optimization problem, outlines the algorithmic framework, and provides theoretical guarantees for the proposed algorithm.
\S\ref{sec:applications} describes the application problems and the datasets used.
\S\ref{sec:num-results} presents numerical results, demonstrating the scalability and performance of our approach.
Finally, \S\ref{sec:conclusions} summarizes our contributions and discusses potential future directions.

\section{Notations}\label{sec:notations}
Vectors are represented by bold-faced lower-case letters (for instance, $\boldsymbol{x}$). Scalars are represented by lower-case letters (for instance, $n$).
We denote $\boldsymbol{e} = \bmat{1, \cdots, 1}^\top$ the vector of all ones and $\boldsymbol{0} = \bmat{0, \cdots, 0}^\top$ the vector of all zeros.
The indexing for a vector $\boldsymbol{x}$ of length $n$ starts from 0 and ends at $n-1$, i.e. $\boldsymbol{x} = \bmat{\boldsymbol{x}_0,\dots,\boldsymbol{x}_{n-1}}^{\top}$.
We use a subscript to denote an element of a vector, i.e. $\boldsymbol{x}_i$. For a vector $\boldsymbol{x}$, $\diag(\boldsymbol{x})$ represents a diagonal matrix with the $i^{th}$ diagonal entry being $\boldsymbol{x}_i$.
$[n]$ represents the set $\{0,1,\dots,n-1\}$. $\boldsymbol{x}_{1:\frac{n}{2}-1}$ is a column vector $\bmat{\boldsymbol{x}_1,\dots,\boldsymbol{x}_{\frac{n}{2}-1}}^\top$.
We denote $\mathcal{M}$ as a subset of the set $[n]$, and $[n]\backslash \mathcal{M}$ denotes the complement set of $\mathcal{M}$ in $[n]$.
$|\mathcal{A}|$ represents the cardinality of set $\mathcal{A}$. For a vector $\boldsymbol{x}$ of length $n$ and $\mathcal{A}\subseteq [n]$, $\boldsymbol{x}_{\mathcal{A}}$ refers to the subvector of $\boldsymbol{x}$ indexed by the set $\mathcal{A}$.
Matrices are represented by upper-case letters (such as, $A$).
If $A$ is a $n\times n$ matrix, then $A_j$ means the $j^{th}$ row of $A$. For two sets $\mathcal{A}, \mathcal{B}\subseteq [n]$, $A_{\mathcal{A}}$ is a submatrix of $A$ containing all the rows indexed by $\mathcal{A}$ and $A_{\mathcal{A}\times \mathcal{B}}$ is a submatrix of $A$ containing rows indexed by $\mathcal{A}$ and columns indexed by $\mathcal{B}$.
We denote $I$ as identity matrix and $0$ as the matrix of all zeros.
Let $a$ be a complex number.
$\overline{a}$ means the conjugate of $a$.
$Re(a)$ refers to the real part of $a$, $Im(a)$ refers to the imaginary part of $a$ and $|a|=\sqrt{Re(a)^2 + Im(a)^2}$.
We note $\boldi = \sqrt{-1}$ the imaginary number.
$\|\boldsymbol{x}\|_1 = \sum_{i=0}^{n-1}|\boldsymbol{x}_i|$ and $\|\boldsymbol{x}\|_2 = \sqrt{\sum_{i=0}^{n-1}|\boldsymbol{x}_i|^2}$ define the usual vectorial norms.
For two vectors $\boldsymbol{x}, \boldy\in\mathbb{R}^n$, $\boldsymbol{x}\geq (>) \boldy$ means $\boldsymbol{x}_i\geq (>) \boldy_i$ for any $i\in [n]$. For two matrices $A, B\in\mathbb{R}^{n\times n}$, $A\succ 0$ means $A$ is a positive definite matrix and $A\succ B$ means $A-B\succ 0$.
For two scalars $a$ and $b$, we say $a = \Theta(b)$ if there exist constants $c_1$ and $c_2$ such that $c_1 b\leq a \leq c_2 b$.

\section{Problem formulation}\label{sec:formulation}
In this section, we detail the formulation of the compressed sensing problem.
We aim at improving the resolution of a signal with missing values, under the assumption that its DFT is sparse.
For simplicity, we focus the presentation of the abstraction on one-dimensional data of even length.
The transformation between a real signal \(\boldsymbol{x} \in \mathbb{R}^n\) and its DFT \(\boldv \in \mathbb{C}^n\) is given by
\begin{equation*}
    \boldv = DFT(\boldsymbol{x}) \quad \iff \quad \boldsymbol{x} = IDFT(\boldv).
\end{equation*}
The DFT is defined as
\begin{equation*}\label{FTdef}
    \boldv_{k} = \frac{1}{\sqrt{n}}\sum_{j=0}^{n-1}\omega^{jk}\boldsymbol{x}_j,\quad\text{for}\quad k = 0,\dots,n-1 \;,
\end{equation*}
where $\omega = e^{-\boldi\frac{2\pi}{n}}$. The IDFT is defined as
\begin{equation*}\label{FTdef2}
    \boldsymbol{x}_{k} = \frac{1}{\sqrt{n}}\sum_{j=0}^{n-1}\overline{\omega}^{jk}\boldv_j,\quad \text{for}\quad k = 0,\dots,n-1.
\end{equation*}
The DFT and IDFT can be represented as linear operators using the following unitary matrix:
$$
C = \frac{1}{\sqrt{n}}
\bmat{
   1      & 1            & 1             & \cdots & 1
\\ 1      & \omega       & \omega^2      & \cdots & \omega^{n-1}
\\ 1      & \omega^2     & \omega^4      & \cdots & \omega^{2n-2}
\\ \vdots & \vdots       & \vdots        & \ddots & \vdots
\\ 1      & \omega^{n-1} & \omega^{2n-2} & \cdots & \omega^{(n-1)^2}
}.
$$
Using this Vandermonde matrix, the transformations can be compactly expressed as
$\boldv = C \boldsymbol{x}$ and $\boldsymbol{x} = C^H \boldv$, where $C^H$ denotes the conjugate transpose of $C$.

Since the signal \(\boldsymbol{x}\) is real, its DFT \(\boldsymbol{v}\) must be conjugate symmetric by definition, which means \(\boldsymbol{v}\) lies in the subspace
\[
\mathcal{F} = \left\{\boldsymbol{v} \in \mathbb{C}^n \mid \boldsymbol{v}_0 \in \mathbb{R}; \, \boldsymbol{v}_{\frac{n}{2}} \in \mathbb{R}; \, \boldsymbol{v}_k = \overline{\boldsymbol{v}_{n-k}} \text{ for } k = 1, \dots, \frac{n}{2}-1 \right\}.
\]
If the length of the signal \(\boldsymbol{x}\) is odd, the DFT \(\boldsymbol{v}\) satisfies
\[
\boldsymbol{v} \in \mathcal{F} := \left\{\boldsymbol{v} \in \mathbb{C}^n \mid \boldsymbol{v}_0 \in \mathbb{R}; \, \boldsymbol{v}_k = \overline{\boldsymbol{v}_{n-k}} \quad \text{for} \quad k = 1, \dots, \frac{n+1}{2}\right\}.
\]
We discuss the odd case here only to point out the slight difference of the subspace definition from the even case; in the sequel we discuss the even case only.

In practice, the observed signal $\boldsymbol{x}$ is not only noisy but also incomplete.
We assume the noisy signal $\widetilde{\boldb}$ follows an additive noise model $\boldsymbol{\widetilde{b}} = \boldsymbol{x}+ \boldsymbol{\widetilde{e}}$, where $\boldsymbol{x}$ represents the true signal, $\boldsymbol{\widetilde{e}}$ denotes the noise. In addition, the noisy signal $\widetilde{\boldb}$ also has missing values. Let $\mathcal{M}$ denote the set of indices corresponding to the missing entries. To simplify notation, we denote the observed signal 
as $\boldb:= \widetilde{\boldb}_{[n]\backslash\mathcal{M}}$, where $\boldb$ contains only the observed values from the noisy signal $\widetilde{\boldb}$.
A common approach to recover the DFT $\boldv$ is by solving the constrained least-squares problem in complex numbers
\begin{equation}\label{complexopt}
    \min_{\boldv \in \mathbb{C}^n}\left\|\boldb - \left(IDFT(\boldv)\right)_{[n]\backslash\mathcal{M}}\right\|^2\quad\text{subject to}\quad\boldv\in \mathcal{F}.
\end{equation}
This optimization problem is challenging due to the feasibility constraint $\boldv\in \mathcal{F}$.
To address this issue, we introduce two one-to-one mappings.
It helps transform \eqref{complexopt} into the real domain and eliminate the feasibility constraint.
The first maps the DFT $\boldv \in \mathbb{C}^n$ to a real vector $\boldbeta\in\mathbb{R}^n$, defined as
\begin{equation}\label{beta}
    \boldbeta = [\boldv_0,~\boldv_{\frac{n}{2}},~\sqrt{2}Re(\boldv_{1:\frac{n}{2}-1}),~\sqrt{2}Im(\boldv_{1:\frac{n}{2}-1})]^{\top}.
\end{equation}
We can also express this mapping as a matrix-vector product $\boldbeta = B\boldv$, where
$$
B = \bmat{
   1 &                    &          &                    &   &                   &         &
\\   &                    &          &                    & 1 &                   &         &
\\   & \phantom{-\boldi}\coefB &          &                    &   &                   &         & \phantom{\boldi}\coefB
\\   &                    &  \ddots  &                    &   &                   & \iddots &
\\   &                    &          & \phantom{-\boldi}\coefB &   & \phantom{\boldi}\coefB &         &
\\   &                    &          & -i\coefB           &   & i\coefB           &         &
\\   &                    &  \iddots &                    &   &                   & \ddots  &
\\   & -i\coefB           &          &                    &   &                   &         & i\coefB
} \in \mathbb{C}^{n \times n}.
$$
Its inverse mapping is given by
\begin{equation}\label{v}
\begin{aligned}
    \boldv_0 &= \boldbeta_0,\quad \boldv_{\frac{n}{2}} = \boldbeta_1, \\
    \boldv_k &= \tfrac{1}{\sqrt{2}}(\boldbeta_{k+1} + \boldi\cdot\boldbeta_{k+\frac{n}{2}})\quad \text{for}\quad k=1,\dots,\tfrac{n}{2}-1,\\
    \boldv_{k} &= \overline{\boldv_{n-k}}\quad \text{for}\quad k=\tfrac{n}{2}+1,\dots,n-1.
\end{aligned}
\end{equation}
Because the matrix $B$ is unitary ($BB^{H} = B^{H\!}B= I_n$), the inverse mapping can also be expressed as a matrix-vector product $\boldv = B^{H\!}\boldbeta$, where $B^{H}$ is the adjoint of $B$.
The mappings ensure $IDFT(\boldv) = A\boldbeta$ where $A \in \mathbb{R}^{n \times n}$ is a real orthogonal matrix.
To verify this, we observe that the $k^{th}$ component of $IDFT(\boldv)$ satisfies
\begin{equation}\label{map}
    \begin{aligned}
        x_k &= \tfrac{1}{\sqrt{n}}{\textstyle\sum_{j=0}^{n-1}}\overline{\omega}^{jk}\boldv_j\\
            &= \tfrac{1}{\sqrt{n}}\left(\boldv_0 + (-1)^k \boldv_{\frac{n}{2}} +{\textstyle\sum_{j=1}^{\frac{n}{2}-1}}\overline{\omega}^{jk}\boldv_j + {\textstyle\sum_{j=\frac{n}{2}+1}^{n-1}}\overline{\omega}^{jk}\boldv_j\right)\\
            &= \tfrac{1}{\sqrt{n}}\left(\boldv_0 + (-1)^k \boldv_{\frac{n}{2}} +{\textstyle\sum_{j=1}^{\frac{n}{2}-1}}\overline{\omega}^{jk}\boldv_j + {\textstyle\sum_{j=\frac{n}{2}+1}^{n-1}}\overline{\omega}^{jk}\overline{\boldv_{n-j}}\right)\\
            &= \tfrac{1}{\sqrt{n}}\left(\boldv_0 + (-1)^k \boldv_{\frac{n}{2}} +{\textstyle\sum_{j=1}^{\frac{n}{2}-1}}\overline{\omega}^{jk}\boldv_j + {\textstyle\sum_{j=1}^{\frac{n}{2}-1}}\overline{\omega}^{(n-j)k}\overline{\boldv_j}\right)\\
            &= \tfrac{1}{\sqrt{n}}\left(\boldv_0 + (-1)^k \boldv_{\frac{n}{2}} +{\textstyle\sum_{j=1}^{\frac{n}{2}-1}}\overline{\omega}^{jk}\boldv_j + {\textstyle\sum_{j=1}^{\frac{n}{2}-1}}\omega^{jk}\overline{\boldv_j}\right)\\
            &= \tfrac{1}{\sqrt{n}}\left(\boldv_0 + (-1)^k \boldv_{\frac{n}{2}} +{\textstyle\sum_{j=1}^{\frac{n}{2}-1}}2Re(\overline{\omega}^{jk}\boldv_j)\right)\\
            &= \tfrac{1}{\sqrt{n}}\left(\boldv_0 + (-1)^k \boldv_{\frac{n}{2}} +{\textstyle\sum_{j=1}^{\frac{n}{2}-1}}2Re(\overline{\omega}^{jk})Re(\boldv_j) -{\textstyle\sum_{j=1}^{\frac{n}{2}-1}} 2Im(\overline{\omega}^{jk})Im(\boldv_j)\right)\\
            &= \tfrac{1}{\sqrt{n}}\left(\boldbeta_0 + (-1)^k \boldbeta_1 +{\textstyle\sum_{j=1}^{\frac{n}{2}-1}}\sqrt{2} Re(\overline{\omega}^{jk})\boldbeta_{j+1} -{\textstyle\sum_{j=1}^{\frac{n}{2}-1}} \sqrt{2}Im(\overline{\omega}^{jk})\boldbeta_{j+\tfrac{n}{2}}\right).\\
            &= A_k\boldbeta.
    \end{aligned}
\end{equation}
$A$ is a real matrix because, for any row $k$, the coefficients of $\boldbeta$ are multiplied by real coefficients. Since $A = C^{H\!} B^H$, where $B$ and $C$ are unitary matrices, $A$ is orthogonal.
We construct analogous mappings such that $\text{IDFT}(\boldv) = A \boldbeta$ for both the two- and three-dimensional problems.
A detailed explanation of these mappings is available in the documentation of the Julia package \texttt{CompressedSensingIPM.jl}.


Next, we define two submatrices of $A$ encoding the missing values
\begin{equation}\label{MMperp}
    M = A_{\mathcal{M}}\text{ and }M_{\perp} = A_{[n]\backslash\mathcal{M}}.
\end{equation}
The complex optimization problem \eqref{complexopt} becomes equivalent to an unconstrained quadratic optimization problem
\begin{equation*}
    \min_{\boldbeta\in\mathbb{R}^n}\left\|\boldb - M_{\perp}\boldbeta\right\|^2.
\end{equation*}
We observe that, with our mappings, sparsity in $\boldv$ induces sparsity in $\boldbeta$, and vice versa.
This mutual relationship ensures that imposing a sparse structure on one of these variables automatically constrains the other to have a matching sparsity pattern.
To enforce the sparse structure in the solution $\boldbeta$, we introduce an $\ell_1$ penalty term on $\boldbeta$ to derive a least absolute shrinkage and selection operator (LASSO) problem
\begin{equation}\label{eq:realopt}
    \boldsymbol{\beta^{\star}} = \underset{\boldbeta\in\mathbb{R}^n}{\arg\min}~\tfrac{1}{2}\|\boldb-M_{\perp}\boldbeta\|^2_2~+~\lambda\|\boldbeta\|_1,
\end{equation}
where $\lambda\in\mathbb{R}_+$ is the penalty parameter.

\subsection{Matrix-free operators}

While the dense matrix $M_{\perp}$ can be directly defined from $A$, it is unnecessary to explicitly store \( M_{\perp} \) in practice.
The computational cost and memory requirements of a dense matrix-vector multiplication are significantly higher than those of an operator-vector product, where the operator is implemented as an FFT call combined with a mapping.
For \( \boldbeta\in\mathbb{R}^n \) and \( \boldb\in\mathbb{R}^{n-|\mathcal{M}|} \), the matrix-vector multiplications \( M_{\perp}\boldbeta \) and \( M_{\perp}^\top \boldb \) can be efficiently computed using real-to-complex DFT and complex-to-real IDFT operations, combined with the appropriate mapping.
This approach significantly reduces memory usage and improves efficiency due to the \( \mathcal{O}(n \log n) \) complexity of FFT-based computations.
For $M_{\perp}\boldbeta$, using the mapping in \eqref{map}, we have
\begin{equation}
    M_{\perp}\boldbeta = (A\boldbeta)_{[n]\backslash\mathcal{M}} = (IDFT(\boldv))_{[n]\backslash\mathcal{M}},
\end{equation}
where $\boldv$ is mapped from $\boldbeta$ via \eqref{v}. For $M_{\perp}^\top \boldb$, the orthogonality of $A$ ensures the existence of a unique vector $\boldxi\in \mathbb{R}^{n}$ such that
\begin{equation}\label{xi}
    M_{\perp}\boldxi = \boldb, \quad M\boldxi = \boldsymbol{0}.
\end{equation}
Moreover, by the definition of $M$ and $M_{\perp}$ in \eqref{MMperp}, we observe that
\begin{equation}
    \boldxi = I_n \boldxi = A^{\top\!\!} A \boldxi = M_{\perp}^{\top\!} M_{\perp}\boldxi +  M^{\top\!}M\boldxi = M_{\perp}^\top \boldb + M^{\top}\boldsymbol{0} = M_{\perp}^\top \boldb.
\end{equation}
Thus, the problem is transferred to computing $\boldxi$.
To achieve this, we construct a vector $\widehat{\boldb}\in\mathbb{R}^n$ with $\widehat{\boldb}_{\mathcal{M}}=\boldzero$ and $\widehat{\boldb}_{[n]\backslash\mathcal{M}} = \boldb$. We further note that (\ref{xi}) is equivalent to
\begin{equation}
    \boldxi = A^{\top}\widehat{\boldb} = DFT(\boldu),
\end{equation}
where $\boldu\in \mathcal{F}$ is mapped from the real vector $\widehat{\boldb}$ following the mapping defined in \eqref{v}.

\subsection{Karush-Kuhn-Tucker stationary conditions}
\label{sec:ipm}
We linearize the $\ell_1$ penalty in the problem~\eqref{eq:realopt} using an elastic reformulation.
It amounts to introduce additional decision variables $\boldsymbol{z} \in \mathbb{R}^n$,
solution of the new optimization problem:
\begin{equation}
  \label{eq:problem}
  \begin{aligned}
    \min_{\boldbeta \in \mathbb{R}^{n}, \boldz \in \mathbb{R}^{n}} \; & \tfrac{1}{2}\|\boldb - \Mperp \boldbeta \|^2 + \lambda \boldsymbol{e}^{\top\!} \boldz \\
    \text{subject to} \; & - \boldz \leq \boldbeta \leq \boldz.
  \end{aligned}
\end{equation}
We rewrite the problem~\eqref{eq:problem} in standard form
by introducing two positive slack variables $\bolds_1, \bolds_2 \in \mathbb{R}^{n}_+$.
Problem~\eqref{eq:problem} becomes equivalent to the quadratic program (QP):
\begin{equation}
  \label{eq:slackproblem}
  \begin{aligned}
    \min_{\boldbeta, \boldz, \boldsymbol{s}_1, \boldsymbol{s}_2} \; & \tfrac{1}{2}\|\boldb - \Mperp\boldbeta\|^2~+~\lambda \boldsymbol{e}^{\top\!} \boldz \; , \\
    \text{subject to} \quad & \boldz + \boldbeta - \bolds_1 = \boldzero \; , \\
                            & \boldz - \boldbeta - \bolds_2 = \boldzero \; ,  \\
                            & (\bolds_1, \bolds_2) \geq \boldzero \; .
  \end{aligned}
\end{equation}
We denote by $\bolds = (\bolds_1, \bolds_2)$ the slack variables,
$\boldy = (\boldy_1, \boldy_2) \in \mathbb{R}^{2n}$ the
multipliers associated to the two equality constraints, and $\boldnu = (\boldnu_1, \boldnu_2) \in \mathbb{R}^{2n}$
the multipliers associated to the two bound constraints.
We define the Lagrangian of \eqref{eq:slackproblem} as:
\begin{multline}
  \label{eq:lagrangian}
  L(\boldbeta, \boldz, \bolds, \boldy, \boldnu)\\
  =
  \tfrac{1}{2}
  \|\boldb - \Mperp\boldbeta\|^2 + \lambda \boldsymbol{e}^{\top\!} \boldz
  + \boldy_1^\top (\bolds_1 - \boldbeta - \boldz)
  + \boldy_2^\top (\bolds_2 - \boldz + \boldbeta)
  -\boldnu_1^\top \bolds_1 - \boldnu_2^\top \bolds_2 \; .
\end{multline}

The KKT conditions of problem \eqref{eq:slackproblem} are
\begin{equation}
  \label{eq:kkt}
  \begin{aligned}
    & \Mperp^\top (\Mperp\boldbeta - \boldb) - \boldy_1 + \boldy_2 = \boldzero \\
    & \lambda \bolde - \boldy_1 - \boldy_2 = \boldzero \\
    & \boldy_1 - \boldnu_1 = \boldzero \\
    & \boldy_2 - \boldnu_2 = \boldzero \\
    &\boldz + \boldbeta  - \bolds_1 = \boldzero \\
    &\boldz -\boldbeta - \bolds_2 = \boldzero \\
    & \mathbf{0} \leq \bolds_1 \perp \boldnu_1 \geq \boldzero \\
    & \mathbf{0} \leq \bolds_2 \perp \boldnu_2 \geq \boldzero
  \end{aligned}
\end{equation}
As the Problem~\eqref{eq:slackproblem} is convex, the primal-dual solution
$(\boldbeta^\star, \boldz^\star, \bolds^\star, \boldy^\star, \boldnu^\star)$
is an optimal solution of \eqref{eq:slackproblem} if and only if it satisfies
the KKT conditions~\eqref{eq:kkt}.

We analyze further the well-posedness of \eqref{eq:kkt} by looking at the constraints
qualifications. We note $g(\boldbeta, \boldz, \bolds) := (\boldz + \boldbeta - \bolds_1, \boldz - \boldbeta - \bolds_2)^\top$
the equality constraints, and $h(\boldbeta, \boldz, \bolds) := s$ the inequality constraints.
The active set is defined as
\begin{equation}
    \mathcal{B} = \{ i \in [2n] \; | \; s_i = 0 \} \; ,
\end{equation}
and we note $h_\mathcal{B}(\cdot)$ the vector containing the components  of $h(\cdot)$ for $i \in \mathcal{B}$.
The active set satisfying strict complementarity is defined as: $\mathcal{B}_+ = \{ i \in \mathcal{B} \; | \; \nu_i > 0 \; \text{for some} \; \boldnu \;
\text{satisfying \eqref{eq:kkt}} \}$.

Observe that if both $s_{1i} = 0$ and $s_{2i} = 0$ hold, then the constraints imply $\beta_i = z_i = 0$ and the variable
$\beta_i$ becomes inactive.  A contrario, we define the set of active variables:
\begin{equation}
\label{eq:activevariables}
    \mathcal{A} = \{ i \in [n] \; | \; |\beta_i| > 0 \} \; .
\end{equation}
Observe that $z_i >0$ for all $i \in \mathcal{A}$. We note the active submatrix keeping only
the active columns in $\Mperp$ as $N_\mathcal{A} := (\Mperp)_{:,\mathcal{A}}$.

\begin{assumption}[Linear independence constraint qualification (LICQ)]
\label{assum:LICQ}
We say that the linear independence constraint qualification (LICQ) holds at $(\boldbeta, \boldz, \bolds)$ if
the set of active constraint gradients $\{\nabla g_i(\boldbeta, \boldz, \bolds) \}_{i \in [2n]} \cup \{\nabla h_i(\boldbeta, \boldz, \bolds) \}_{i \in \mathcal{B}}$ is linearly independent.
\end{assumption}
The condition LICQ ensures that there exists an unique dual multiplier $(\boldy, \boldnu)$ satisfying
the KKT conditions \eqref{eq:kkt} at the solution. Similarly, the following condition gives the
uniqueness of the primal solution.

\begin{assumption}[Second-order sufficient condition (SOSC)]
\label{assum:SOSC}
We say that the second-order sufficient condition (SOSC) holds at $(\boldbeta, \boldz, \bolds)$ if
for all multipliers $(\boldy, \boldnu)$ satisfying \eqref{eq:kkt} there exists a constant
$\alpha > 0$ such that
\begin{subequations}
\label{eq:sosc}
\begin{equation}
\label{eq:ssosc}
   \boldd^\top \big\{ \nabla^2_{(\boldbeta, \boldz, \bolds)} L(\boldbeta, \boldz, \bolds, \boldy, \boldnu)\big\} \boldd \geq \alpha \|\boldd \|^2 \; ,
\end{equation}
for all directions $\boldd \in \mathbb{R}^{4n}$ satisfying
\begin{equation}
\label{eq:criticalcone}
  \begin{aligned}
    & \nabla g_i(\boldbeta, \boldz, \bolds)^\top \boldd  = 0 \quad \forall i \in [2n]\;, \\
    & \nabla h_i(\boldbeta, \boldz, \bolds)^\top \boldd  = 0 \quad \forall i \in \mathcal{B}_+ \; , \\
    & \nabla h_i(\boldbeta, \boldz, \bolds)^\top \boldd  \geq 0 \quad \forall i \in \mathcal{B} \setminus \mathcal{B}_+ \; .
  \end{aligned}
\end{equation}
\end{subequations}
\end{assumption}
The strict complementarity condition ensures that $\mathcal{B} = \mathcal{B}_+$.
\begin{assumption}[Strict complementarity condition (SCS)]
\label{assum:SCS}
    We say that the primal-dual solution $(\bolds^\star, \boldnu^\star)$
    satisfies strict complementarity if $\bolds_i^\star + \boldnu_i^\star > 0$ for all $i \in [2n]$.
\end{assumption}

In the following proposition, we prove that the QP~\eqref{eq:slackproblem} satisfies LICQ
and SOSC under certain assumptions.

\begin{proposition}
\label{prop:wellposedness}
The problem~\eqref{eq:slackproblem} satisfies LICQ. If SCS hold and
the active submatrix $N_{\mathcal{A}}$ is full row rank, then
the problem \eqref{eq:slackproblem} also satisfies SOSC.
\end{proposition}
\begin{proof}
The constraints of \eqref{eq:slackproblem} are affine, and their Jacobian is
\begin{equation}
J = \begin{bmatrix}
   \phantom{-} I & I & \!\!          -  I & \!\! \phantom{-} 0
\\          -  I & I & \!\! \phantom{-} 0 & \!\!          -  I
\\ \phantom{-} 0 & 0 & \!\! \phantom{-} I & \!\! \phantom{-} 0
\\ \phantom{-} 0 & 0 & \!\! \phantom{-} 0 & \!\! \phantom{-} I
\end{bmatrix} \; .
\end{equation}
We observe that the Jacobian $J$ is full row-rank, implying the gradient of the constraints
are linearly independent and LICQ holds.

Let $\boldd = (\boldd_\beta, \boldd_z, \boldd_s) \in \mathbb{R}^{4n}$ satisfying \eqref{eq:criticalcone}. First, note  that
\begin{equation}
\label{eq:soscproof1}
   \boldd^\top \big\{ \nabla^2_{(\boldbeta, \boldz, \bolds)} L(\boldbeta, \boldz, \bolds, \boldy, \boldnu)\big\} \boldd  =
   \boldd_\beta \Mperp^\top \Mperp \boldd_\beta \; .
\end{equation}
Under SCS, the conditions \eqref{eq:criticalcone} imply that
\begin{equation}
\begin{aligned}
& (\boldd_z)_i + (\boldd_\beta)_i - (\boldd_{s_1})_i = 0  \quad \forall i \in [n] \; , \\
& (\boldd_z)_i - (\boldd_\beta)_i - (\boldd_{s_2})_i = 0  \quad \forall i \in [n] \; , \\
& (\boldd_s)_i = 0 \quad \forall i \in \mathcal{B} \; .
\end{aligned}
\end{equation}
Suppose that the variable $i$ is inactive: $i \in [n] \setminus \mathcal{A}$. Then we
have both $\bolds_{1i} = 0$ and $\bolds_{2i} = 0$, implying $(\boldd_{s_1})_i = (\boldd_{s_2})_i = 0$.
Hence, the conditions \eqref{eq:criticalcone} implies that $(\boldd_\beta)_i = 0$ for all $i \in [n] \setminus \mathcal{A}$.
Hence, if $\boldd$ satisfies \eqref{eq:criticalcone}, the equation \eqref{eq:soscproof1} simplifies as
\begin{equation}
\label{eq:soscproof2}
   \boldd^\top \big\{ \nabla^2_{(\boldbeta, \boldz, \bolds)} L(\boldbeta, \boldz, \bolds, \boldy, \boldnu)\big\} \boldd  =
   \boldd_\beta \Mperp^\top \Mperp \boldd_\beta =
   (\boldd_\beta)_\mathcal{A} N_\mathcal{A}^\top N_\mathcal{A} (\boldd_\beta)_\mathcal{A} \; .
\end{equation}
Note $\sigma_1 = \sigma_{min}(N_\mathcal{A})$ the smallest singular value of the active submatrix $N_\mathcal{A}$.
If we suppose $N_\mathcal{A}$ is full row rank, we have $\sigma_1 > 0$, and
\begin{equation}
    \label{eq:soscproof3}
   (\boldd_\beta)_\mathcal{A} N_\mathcal{A}^\top N_\mathcal{A} (\boldd_\beta)_\mathcal{A} \geq \sigma_1^2 \| (\boldd_\beta)_\mathcal{A} \|^2  \; .
\end{equation}
For all $i \in [n] \setminus \mathcal{A}$, we have that $(\boldd_\beta)_i = (\boldd_z)_i =(\boldd_{s_1})_i = (\boldd_{s_2})_i =  0$.
On the active set $i \in \mathcal{A}$, we have either $(\boldd_{s_1})_i = 0$ or $(\boldd_{s_2})_i = 0$.
Thus in turn implies: $|(\boldd_{z})_i| = |(\boldd_{\beta})_i|$.
As $(\boldd_{s_1})_i + (\boldd_{s_2})_i = 2 (\boldd_\beta)_i$, we deduce
that $\| \boldd_z \| = \| \boldd_\beta \|$ and $\| \boldd_s \| = 2 \| \boldd_\beta \|$.
As $\| \boldd_\beta \| = \| (\boldd_\beta)_\mathcal{A} \|$,
we can bound the expression in the right-hand-side of \eqref{eq:soscproof3} below by a multiple of $\| \boldd \|^2$.
Hence concluding the proof.
\end{proof}

Proposition~\ref{prop:wellposedness} requires both SCS and a full-rank active submatrix $N_\mathcal{A}$.
On the one hand, SCS may fail to hold for some values of the penalty term $\lambda$, so this assumption is restrictive.
On the other hand as $\lambda$ is either fixed or sampled; we believe it is reasonable to assume it for the instance of the problem we are solving.
Also, the matrix $N_\mathcal{A}$ is likely to be full row-rank is the solution $\beta$ is very sparse,
meaning we keep only a few columns in $N_\mathcal{A}$.

\subsection{Primal-dual interior-point method}
We solve~\eqref{eq:kkt} using a primal-dual interior-point method.
For a barrier term $\mu > 0$, the KKT equations~\eqref{eq:kkt}
are reformulated as a smooth system of equations using a homotopy method~\cite{wright1997primal}.
For a fixed barrier term $\mu$, the primal-dual interior-point method
solves the following system of nonlinear equations:
\begin{equation}
  \label{eq:kktmu}
  \begin{aligned}
    & \Mperp^\top ( \Mperp\boldbeta - \boldb) - \boldy_1 + \boldy_2 = \boldzero \; ,\\
    & \lambda \bolde - \boldy_1 - \boldy_2 = \boldzero \; ,\\
    & \boldy_1 - \boldnu_1 = \boldzero \; ,\\
    & \boldy_2 - \boldnu_2 = \boldzero \; ,\\
    &\boldz + \boldbeta  - \bolds_1 = \boldzero \; ,\\
    &\boldz - \boldbeta  - \bolds_2 = \boldzero \; ,\\
    & S_1 V_1 \boldsymbol{e} = \mu \boldsymbol{e} \; , \quad (\bolds_1, \boldnu_1) > \boldzero  \; ,\\
    & S_2 V_2 \boldsymbol{e} = \mu \boldsymbol{e} \; , \quad (\bolds_2, \boldnu_2) > \boldzero \;,
  \end{aligned}
\end{equation}
where we have defined the diagonal matrices
\begin{equation}
S_1 = \diag(\bolds_1) \;, \quad
S_2 = \diag(\bolds_2) \; , \quad
V_1 = \diag(\boldnu_{1}) \;  \quad
V_2 = \diag(\boldnu_{2}) \; .
\end{equation}
Once the iterate is sufficiently close to the central path, the barrier parameter $\mu$ is decreased.
For feasible primal-dual interior points not on the central path, it is customary to define $\mu$ with the duality measure:
\begin{equation}
  \mu = \frac{1}{2n} (\boldnu_1^{\top\!} \bolds_1 + \boldnu_2^{\top\!} \bolds_2);
\end{equation}
observe that this identity holds intrinsically when \eqref{eq:kktmu} holds.
As $\mu \to 0$, we recover the original KKT conditions~\eqref{eq:kkt}.

\begin{definition}[Centrality conditions]
For parameters $C > 0$ and $\gamma \in (0, 1)$, we say that the primal-dual
iterate $(\boldbeta, \boldz, \bolds, \boldy, \boldnu)$  satisfies the \emph{centrality conditions} if
\begin{equation}
  \label{eq:centralpath}
  \begin{aligned}
    & \| \nabla_{(\boldbeta, \boldz, \bolds)} L(\boldbeta, \boldz, \bolds, \boldy, \boldnu) \| \leq C\mu \; ,  \\
    & \| \boldz - \boldbeta - \bolds_1 \| \leq C \mu,\\
    & \| \boldz + \boldbeta - \bolds_2 \| \leq C \mu,\\
    & (\bolds, \boldnu) > 0 \; , \quad s_i \nu_i \geq \gamma \mu \; , \quad \forall i =1 , \cdots, 2n.
  \end{aligned}
\end{equation}
\end{definition}

\subsection{KKT linear systems}
The nonlinear system~\eqref{eq:kktmu} is solved iteratively by a Newton method. The method ensures that the iterates have enough centrality; that is,
by adjusting the right hand side of the Newton system and line search to ensure that \eqref{eq:centralpath} hold; see \cite{wright1997primal,wright2001effects}.

\subsubsection{Augmented KKT system}
 At each iteration, the primal-dual interior-point method solves the
following {\it augmented KKT system}:
\begin{equation}
  \label{eq:augkkt}
  \begin{bmatrix}
    \Mperp^\top \Mperp & \!\!\phm 0 & \phm 0        & \phm 0        & -I     &\phm I
    \\ \!\!\phm 0      & \!\!\phm 0 & \phm 0        & \phm 0        & -I     & -I
    \\ \!\!\phm 0      & \!\!\phm 0 & \phm \Sigma_1 & \phm 0        & -I     & \phm 0
    \\ \!\!\phm 0      & \!\!\phm 0 & \phm 0        & \phm \Sigma_2 & \phm 0 & -I
    \\ \!\!-I          & \!\!-I     & -I            & \phm 0        & \phm 0 & \phm 0
    \\ \!\!\phm I      & \!\!-I     & \phm 0        & -I            & \phm 0 & \phm 0
  \end{bmatrix}
  \begin{bmatrix}
    \Delta\boldbeta \\
    \Delta \boldz \\
    \Delta \bolds_1 \\
    \Delta \bolds_2 \\
    \Delta \boldy_1 \\
    \Delta \boldy_2
  \end{bmatrix}
  = \begin{bmatrix}
   \boldr_1 \\
   \boldr_2 \\
   \boldr_3 \\
   \boldr_4 \\
   \boldr_5 \\
   \boldr_6
  \end{bmatrix}
\end{equation}
with
\begin{equation}
\label{eq:defsigma}
\Sigma_1 := S_1^{-1} V_1 \quad \text{and}  \quad \Sigma_2 := S_2^{-1} V_2 \; .
\end{equation}
By construction, $\Sigma_1$ and $\Sigma_2$ are two positive definite diagonal matrices: $(\Sigma_1)_{ii} > 0$ and $(\Sigma_2)_{ii} > 0$ for all $i=1,\dots,n$.
The right-hand side is given as $\boldr_1 = \Mperp^\top (\boldb - \Mperp\boldbeta) - \boldy_1 + \boldy_2$, $\boldr_2 = \lambda \bolde - \boldy_1 - \boldy_2$, $\boldr_3 = \boldy_1 - \boldnu_1$, $\boldr_4 = \boldy_2 - \boldnu_2$, $\boldr_5 = \boldz + \boldbeta  - \bolds_1$, and $\boldr_6 = \boldz + \boldbeta  - \bolds_2$.

\subsubsection{Condensed KKT system}
We reduce the KKT system~\eqref{eq:augkkt} by exploiting its structure.
The first step is to remove the slack variables. We note that
\eqref{eq:augkkt} implies that
\begin{equation}
  \Delta \bolds_1 = \Sigma_1^{-1} (\boldr_3 + \Delta \boldy_1), \quad
  \Delta \bolds_2 = \Sigma_2^{-1} (\boldr_4 + \Delta \boldy_2).
\end{equation}
As such, \eqref{eq:augkkt} rewrites equivalently as a $4 \times 4$ system:
\begin{equation}
  \label{eq:kktintermediate}
  \begin{bmatrix}
       \Mperp^\top \Mperp & \!\!\phm 0 & \!\!\!\!-I            & \!\!\!\!\!\!\!\phm I
    \\ \!\!\phm 0         & \!\!\phm 0 & \!\!\!\!-I            & \!\!\!\!\!\!\!-I
    \\ \!\!-I             & \!\!-I     & \!\!~~~-\Sigma_1^{-1} & \!\!\!\!\!\!\!\phm 0
    \\ \!\!\phm I         & \!\!-I     & \!\!\!\!\phm 0        & \!-\Sigma_2^{-1}
  \end{bmatrix}
  \begin{bmatrix}
    \Delta\boldbeta \\
    \Delta \boldz \\
    \Delta \boldy_1 \\
    \Delta \boldy_2
  \end{bmatrix}
  = \begin{bmatrix}
   \boldr_1 \\
   \boldr_2 \\
   \boldr_5 + \Sigma_1^{-1} \boldr_3 \\
   \boldr_6 + \Sigma_2^{-1} \boldr_4
  \end{bmatrix}.
\end{equation}
We can remove $(\Delta \boldy_1, \Delta \boldy_2)$ from \eqref{eq:kktintermediate} by noting that
\begin{equation}
  \Delta \boldy_1 = \Sigma_1 (-\Delta \boldbeta - \Delta \boldz - \boldr_5 - \Sigma_1^{-1} \boldr_3), \quad
  \Delta \boldy_2 = \Sigma_2 ( \Delta \boldbeta - \Delta \boldz - \boldr_6 - \Sigma_2^{-1} \boldr_4).
\end{equation}
We obtain the $2 \times 2$ condensed KKT system:
\begin{equation}
  \label{eq:condensedkkt}
  K := \begin{bmatrix}
    \Mperp^\top \Mperp + \Lambda_1 & \Lambda_2 \\
    \Lambda_2 & \Lambda_1
  \end{bmatrix}
  \begin{bmatrix}
    \Delta \boldbeta \\
    \Delta \boldz
  \end{bmatrix}
  =
  \begin{bmatrix}
    \boldr_\beta \\ \boldr_c
  \end{bmatrix},
\end{equation}
where we define the diagonal matrices
\begin{equation}
\Lambda_1 := \Sigma_1 + \Sigma_2 \quad \text{and} \quad \Lambda_2 = \Sigma_1 - \Sigma_2 \; .
\end{equation}
The right-hand-sides are defined respectively as:
\begin{equation*}
    \begin{aligned}
        \boldr_\beta &:= \boldr_1 - \Sigma_1 (\boldr_5 + \Sigma_1^{-1} \boldr_3) + \Sigma_2 (\boldr_6 + \Sigma_2^{-1} \boldr_4)
                  = \boldr_1 - \boldr_3 + \boldr_4 - \Sigma_1 \boldr_5 + \Sigma_2 \boldr_6,\\
        \boldr_c &:= \boldr_2 - \Sigma_1 (\boldr_5 + \Sigma_1^{-1} \boldr_3) - \Sigma_2 (\boldr_6 + \Sigma_2^{-1} \boldr_4)
              = \boldr_2  - \boldr_3 - \boldr_4 - \Sigma_1 \boldr_5 - \Sigma_2 \boldr_6.
    \end{aligned}
\end{equation*}
The choice of the barrier parameter $\mu$ depends on the IPM implementation we are using,
but the condensed KKT structure~\eqref{eq:condensedkkt} remains the same.

\subsection{Krylov methods and preconditioners}
We aim at solving \eqref{eq:condensedkkt} using a matrix-free linear solving to avoid explicitly materializing $\Mperp$.
For that we can use Krylov methods such as \textsc{CG}, \textsc{CR} \cite{hestenes-stiefel-1952}, or \textsc{CAr} \cite{montoison2023minares} dedicated to symmetric positive definite sytems.
First, we look at the behavior of the diagonal matrices $\Lambda_1$ and $\Lambda_2$ close to optimality.
We note $(\boldbeta^\star, \boldz^\star, \bolds^\star)$ the primal solution and $(\boldy^\star, \boldnu^\star)$ the dual solution of \eqref{eq:slackproblem}.
Note that the KKT conditions imply that $\boldy_1^\star = \boldnu_1^\star$ and $\boldy_2^\star = \boldnu_2^\star$,
hence the bound multipliers $(\boldnu_1^\star, \boldnu_2^\star)$ satisfy
\begin{equation}
  \left\{
  \begin{aligned}
    & \Mperp^\top (\Mperp \boldbeta^\star - \boldb^\star) - \boldnu_1^\star + \boldnu_2^\star = \boldzero, \\
    & \lambda \bolde- \boldnu_1^\star - \boldnu_2^\star = \boldzero.
  \end{aligned}
  \right.
\end{equation}
We introduce the three disjoint sets:
\begin{equation}
   \mathcal{I}_+ = \{ i  \; | \; \boldbeta_i^\star > 0 \} \; , \quad
   \mathcal{I}_- = \{ i  \; | \; \boldbeta_i^\star < 0  \} \; , \quad
   \mathcal{I}_0 = \{ i  \; | \; \boldbeta_i^\star = 0 \} \; .
\end{equation}
Note that $\mathcal{I}_+ \cup \mathcal{I}_- \cup \mathcal{I}_0 = \{1, \cdots, n\}$ and the set of active
variables~\eqref{eq:activevariables} satisfies $\mathcal{A} = \mathcal{I}_+ \cup \mathcal{I}_-$.

Supposing strict complementarity hold,  we get the following relations:
\begin{equation}
  \begin{aligned}
  & i \in \mathcal{I}_+ \implies \boldnu_{1,i}^\star = 0 \; , \; \boldnu_{2,i}^\star > 0 && \implies \boldnu_{2,i}^\star = \lambda \; , \\
  & i \in \mathcal{I}_- \implies \boldnu_{1,i}^\star > 0 \; , \; \boldnu_{2,i}^\star  = 0  && \implies \boldnu_{1,i}^\star = \lambda \; , \\
  & i \in \mathcal{I}_0 \implies \boldnu_{1,i}^\star > 0 \; , \; \boldnu_{2,i}^\star > 0   && \implies \boldnu_{1,i}^\star + \boldnu_{2,i}^\star = \lambda \; .
  \end{aligned}
\end{equation}

\begin{proposition}
 \label{prop:wrightspeed}
  Let $(\boldbeta^\star, \boldz^\star, \bolds^\star, \boldy^\star, \boldnu^\star)$ be a primal-dual
  stationary solution of~\eqref{eq:kkt} satisfying the strict complementarity condition.
  Suppose that  the active submatrix $N_\mathcal{A}$ is full row rank.
  Then, for all the iterates satisfying the centrality conditions~\eqref{eq:centralpath}  we have that:
  \begin{equation}
  \label{eq:convsnu}
  \begin{aligned}
  & i \in \mathcal{I}_+ \implies \bolds^\star_{1,i} = \Theta(1) \; , \; \boldnu_{1,i}^\star = \Theta(\mu); && \bolds^\star_{2,i} = \Theta(\mu) \; , \; \boldnu_{2,i}^\star = \Theta(1) \; , \\
  & i \in \mathcal{I}_- \implies \bolds^\star_{1,i} = \Theta(\mu) \; , \; \boldnu_{1,i}^\star = \Theta(1); && \bolds^\star_{2,i} = \Theta(1) \; , \; \boldnu_{2,i}^\star = \Theta(\mu) \; , \\
  & i \in \mathcal{I}_0\; \implies \bolds^\star_{1,i} = \Theta(\mu) \; , \; \boldnu_{1,i}^\star = \Theta(1); && \bolds^\star_{2,i} = \Theta(\mu) \; , \; \boldnu_{2,i}^\star = \Theta(1) \; .
  \end{aligned}
\end{equation}
\end{proposition}
\begin{proof}
    Using Proposition~\ref{prop:wellposedness}, the problem \eqref{eq:slackproblem}
    satisfies LICQ and SOSC under our assumptions. We know that LICQ implies
    the less restrictive Mangasarian-Fromovitz Constraint Qualification (MFCQ).
    By applying \cite[Lemma 3.2]{wright2001effects} to $(\bolds, \boldz)$, we get the results listed
    in \eqref{eq:convsnu}, hence concluding the proof.
\end{proof}

As a direct corollary, we get the following results  for the diagonal matrices $\Sigma_1$ and $\Sigma_2$ defined in \eqref{eq:defsigma}.

\begin{corollary}
  \label{prop:convergencespeed}
  Let $(\boldbeta^\star, \boldz^\star, \bolds^\star, \boldy^\star, \boldnu^\star)$ be a primal-dual
  stationary solution of \eqref{eq:kkt} satisfying the strict complementarity condition.
  Suppose that  the active submatrix $N_\mathcal{A}$ is full row rank.
  Then, we get the following speed of convergence
  for all the iterates satisfying the centrality conditions~\eqref{eq:centralpath}:
  \begin{itemize}
    \item For $i \in \mathcal{I}_+$, $(\Sigma_1)_{ii} = \Theta(\mu)$,  $(\Sigma_2)_{ii} = \Theta(\frac{1}{\mu})$;
    \item For $i \in \mathcal{I}_-$, $(\Sigma_1)_{ii} = \Theta(\frac{1}{\mu})$,  $(\Sigma_2)_{ii} = \Theta(\mu)$;
    \item For $i \in \mathcal{I}_0$, $(\Sigma_1)_{ii} = \Theta(\frac{1}{\mu})$,  $(\Sigma_2)_{ii} = \Theta(\frac{1}{\mu})$.
  \end{itemize}
  In all cases, we get $(\Lambda_1)_{ii} = \Theta(\frac{1}{\mu})$.
\end{corollary}

To improve the convergence of Krylov methods on KKT systems, we study the preconditioner
\begin{equation}
  P :=
  \begin{bmatrix}
    I + \Lambda_1 & \Lambda_2 \\
    \Lambda_2 & \Lambda_1
  \end{bmatrix}.
\end{equation}
We note that $P$ is very similar to the preconditioner introduced earlier in \cite{Fountoulakis_2013}; the only difference being
they multiply the identity matrix by a scalar dependent on the percentage of missing data.

We define the following diagonal matrices:
\begin{equation}
\label{eq:defBD}
  B := \big(\Lambda_1 - \Lambda_2 (I + \Lambda_1)^{-1} \Lambda_2\big) \; , \quad
D := \big(\Lambda_1 (I + \Lambda_1) - \Lambda_2^2 \big) \; .
\end{equation}
As all the matrices appearing in $B$ are diagonal, we have
$B = (I + \Lambda_1)^{-1} \big(\Lambda_1(I + \Lambda_1) - \Lambda_2^2\big) = (I + \Lambda_1)^{-1} D$.
We explicit further the structure of the matrices $B$ and $D$.
\begin{proposition}
  \label{prop:exprbd}
  The matrices $B$ and $D$ are positive definite, and $D = \Sigma_1 + \Sigma_2 + 4\Sigma_1 \Sigma_2$.
\end{proposition}
\begin{proof}
  We have $\Lambda_1 \succ 0$, implying $(I + \Lambda_1) \succ 0$.
  Additionally, $D = \Lambda_1 + \Lambda_1^2 - \Lambda_2^2 = \Sigma_1 + \Sigma_2 + (\Sigma_1 + \Sigma_2)^2 - (\Sigma_1 - \Sigma_2)^2 = \Sigma_1 + \Sigma_2 + 4 \Sigma_1\Sigma_2$. As $\Sigma_1$ and $\Sigma_2$
  are positive definite, we deduce $\Lambda_1(I + \Lambda_1) - \Lambda_2^2 \succ 0$, concluding
  the proof.
\end{proof}

\begin{proposition}
  The inverse of the preconditioner $P^{-1}$ is given by
  \begin{equation}
    \label{eq:invprecond1}
    P^{-1} = \begin{bmatrix}
      \Lambda_1 D^{-1} & - \Lambda_2 D^{-1} \\
      -D^{-1} \Lambda_2  & B^{-1}
    \end{bmatrix} \; .
  \end{equation}
\end{proposition}
\begin{proof}
Using the block-inversion formula, we have that
\begin{equation*}
  P^{-1} = \begin{bmatrix}
    (I+\Lambda_1)^{-1} + (I+\Lambda_1)^{-1} \Lambda_2 B^{-1} \Lambda_2 (I+\Lambda_1)^{-1} & -(I+\Lambda_1)^{-1} \Lambda_2 B^{-1} \\
    -B^{-1} \Lambda_2 (I+\Lambda_1)^{-1} & B^{-1}
  \end{bmatrix} \; .
\end{equation*}
As all the matrices are diagonal and $B^{-1} = D^{-1} (I + \Lambda_1)$, we get the expressions
of the blocks $(1, 2)$ and $(2, 1)$. Last, observe that as the diagonal matrices commute:
\begin{equation*}
  \begin{aligned}
    (I+\Lambda_1)^{-1}\Big(I +  \Lambda_2 B^{-1} \Lambda_2 (I+\Lambda_1)^{-1} \Big)
   &= (I+\Lambda_1)^{-1}\Big(I +  \Lambda_2 D^{-1} \Lambda_2 \Big) \\
   &= (I+\Lambda_1)^{-1}\big(D +  \Lambda_2^2  \big)D^{-1} \\
   &= \Lambda_1 D^{-1}
  \end{aligned}
\end{equation*}
Concluding the proof.
\end{proof}

According to Proposition~\ref{prop:exprbd}, the $(1, 1)$ block in \eqref{eq:invprecond1} simplifies as
\begin{equation}
\label{eq:11block}
    \Lambda_1 D^{-1} = (\Sigma_1 + \Sigma_2) (\Sigma_1 + \Sigma_2 + 4\Sigma_1 \Sigma_2)^{-1} = (\Sigma_1^{-1} + \Sigma_2^{-1})(\Sigma_1^{-1} + \Sigma_2^{-1} + 4)^{-1} \; .
\end{equation}
As all the diagonal matrices $\Sigma_1$ and $\Sigma_2$ are positive, we get that the diagonal values
$(\Lambda_1 D^{-1})_{ii}$ are bounded between $0$ and $1$ for all $i \in [n]$.

\begin{proposition}
  The preconditioned matrix $P^{-1} K$ satisfies
  \begin{equation}
      \label{eq:precond1}
    P^{-1} K = \begin{bmatrix}
      I + \Lambda_1 D^{-1} \big(\Mperp^\top \Mperp - I \big) & 0 \\
      -D^{-1} \Lambda_2 \big(\Mperp^\top \Mperp -I \big) & I
    \end{bmatrix} \; .
  \end{equation}
\end{proposition}
\begin{proof}
  We compute the four blocks in $P^{-1}K$ one by one using \eqref{eq:invprecond1}.
  First,
  \begin{equation*}
    \begin{aligned}
      (P^{-1} K)_{11} = \Lambda_1 D^{-1} (\Mperp^\top \Mperp + \Lambda_1) - \Lambda_2 D^{-1} \Lambda_2 &= \Lambda_1 D^{-1} \Mperp^\top \Mperp + D^{-1} (\Lambda_1^2 - \Lambda_2^2) \\
                      &= \Lambda_1 D^{-1} \Mperp^\top \Mperp - D^{-1} \Lambda_1 + I,
    \end{aligned}
  \end{equation*}
  where the last simplification comes from the definition of $D$ in \eqref{eq:defBD}. Second,
  \begin{equation*}
    \begin{aligned}
      (P^{-1} K)_{12} &= \Lambda_1 D^{-1} \Lambda_2 - \Lambda_2 D^{-1} \Lambda_1 = 0 \; , \\
      (P^{-1} K)_{22} &= -D^{-1} \Lambda_2^2 + B^{-1} \Lambda_1  = D^{-1} \big((I+\Lambda_1) \Lambda_1  -\Lambda_2^2 + B^{-1} \big) = I \; .
    \end{aligned}
  \end{equation*}
  Finally,
  \begin{equation*}
    \begin{aligned}
    (P^{-1} K)_{21} &= -D^{-1} \Lambda_2 (\Mperp^\top \Mperp + \Lambda_1) + B^{-1} \Lambda_1  \\
                    &= -D^{-1} \Lambda_2 \Mperp^\top \Mperp - D^{-1} \Lambda_2 \Lambda_1 + D^{-1} (I+\Lambda_1) \Lambda_2 \\
                    &= -D^{-1} \Lambda_2 \big( \Mperp^\top \Mperp - I \big) \; .
    \end{aligned}
  \end{equation*}
  Concluding the proof.
\end{proof}

We note that the $(1,2)$ block is $0$ and the block $(2, 2)$ is the identity.
We are interested in the asymptotic behavior of the preconditioned matrix
$P^{-1}K$ as we drive the barrier parameter to $0$.
Without loss of generality, we suppose in the following that the variables
are ordered in the order $(\mathcal{I}_+, \mathcal{I}_-, \mathcal{I}_0)$.

\begin{proposition}
  \label{prop:limitprecond}
  \!Let $(\boldbeta^\star, \boldz^\star, \bolds^\star, \boldy^\star, \boldnu^\star)$ be a primal-dual
  stationary solution of \eqref{eq:kkt} satisfying the strict complementarity condition.
  Suppose that  the active submatrix $N_\mathcal{A}$ is full row rank.
  Then, for $\mu \to 0$ we get the following limit for the $(1,1)$ and $(2,1)$ blocks:
  \begin{subequations}
    \label{eq:precond1conv}
    \begin{equation}
      (P^{-1} K)_{11} \to \begin{bmatrix}
          (\Mperp^\top \Mperp)_{\mathcal{I}_{+}\times \mathcal{I}_{+}} &(\Mperp^\top \Mperp)_{\mathcal{I}_{+}\times \mathcal{I}_{-}} &(\Mperp^\top \Mperp)_{\mathcal{I}_{+}\times \mathcal{I}_{0}}\\
          (\Mperp^\top \Mperp)_{\mathcal{I}_{-}\times \mathcal{I}_{+}} &(\Mperp^\top \Mperp)_{\mathcal{I}_{-}\times \mathcal{I}_{-}} &(\Mperp^\top \Mperp)_{\mathcal{I}_{-}\times \mathcal{I}_{0}}\\
          0 &0 &I_{\mathcal{I}_{0}\times \mathcal{I}_{0}}\\
      \end{bmatrix} \;,
    \end{equation}
    \begin{equation}
      (P^{-1} K)_{21} \to \begin{bmatrix}
          (\Mperp^\top \Mperp - I)_{\mathcal{I}_{+}\times \mathcal{I}_{+}} &(\Mperp^\top \Mperp)_{\mathcal{I}_{+}\times \mathcal{I}_{-}} &(\Mperp^\top \Mperp)_{\mathcal{I}_{+}\times \mathcal{I}_{0}}\\
          (-\Mperp^\top \Mperp)_{\mathcal{I}_{-}\times \mathcal{I}_{+}} &(I - \Mperp^\top \Mperp )_{\mathcal{I}_{-}\times \mathcal{I}_{-}} &(-\Mperp^\top \Mperp)_{\mathcal{I}_{-}\times \mathcal{I}_{0}}\\
          0 &0 & 0
      \end{bmatrix} \;,
    \end{equation}
  \end{subequations}
\end{proposition}
\begin{proof}
  Following Proposition~\ref{prop:exprbd}, $D = \Sigma_1 + \Sigma_2 + 4\Sigma_1\Sigma_2$.
  Using Proposition~\ref{prop:convergencespeed}, we get that
  \begin{equation}
    (\Sigma_1 \Sigma_2)_{\mathcal{I}_+} = \Theta(1)  \;, \quad
    (\Sigma_1 \Sigma_2)_{\mathcal{I}_-} = \Theta(1)  \;, \quad
    (\Sigma_1 \Sigma_2)_{\mathcal{I}_0} = \Theta\left(\frac{1}{\mu^2}\right)  \;.
  \end{equation}
  As a consequence, we deduce that, as $\mu \to 0$,
  \begin{itemize}
    \item For $i \in \mathcal{I}_+$, $(\Lambda_1 D^{-1})_{ii} \to 1$ and $(\Lambda_2 D^{-1})_{ii} \to -1$;
    \item For $i \in \mathcal{I}_-$, $(\Lambda_1 D^{-1})_{ii} \to 1$ and $(\Lambda_2 D^{-1})_{ii} \to 1$;
    \item For $i \in \mathcal{I}_0$, $(\Lambda_1 D^{-1})_{ii} \to 0$ and $(\Lambda_2 D^{-1})_{ii} \to 0$.
  \end{itemize}
  Using the formula of the preconditioner given in \eqref{eq:precond1},
  we get the two limits listed in \eqref{eq:precond1conv}.
\end{proof}

The repartition of the eigenvalues in the preconditioned matrix $P^{-1} K$ depends on the sparsity of the solution $\boldbeta^\star$.
Using Proposition~\ref{prop:limitprecond} together with \eqref{eq:precond1} and \eqref{eq:precond1conv}, we get that the eigenvalue
$1$ has multiplicity $2n - \ell$, where $\ell$ is the number of nonzeroes elements in the sparse
vector $\boldbeta^\star$: $\ell = |\mathcal{I}_+| + |\mathcal{I}_-| = | \mathcal{A} |$.
The sparser the solution, the better the preconditioner.
Going one step further, we note the matrix:
\begin{equation}
   Q^\star =
   \begin{bmatrix}
      (\Mperp^\top \Mperp)_{\mathcal{I}_{+}\times \mathcal{I}_{+}} &(\Mperp^\top \Mperp)_{\mathcal{I}_{+}\times \mathcal{I}_{-}} \\
      \vspace{-0.1cm} & \vspace{-0.1cm} \\
      (\Mperp^\top \Mperp)_{\mathcal{I}_{-}\times \mathcal{I}_{+}} &(\Mperp^\top \Mperp)_{\mathcal{I}_{-}\times \mathcal{I}_{-}}
  \end{bmatrix} \;,
\end{equation}
Let  $\lambda_1 = \lambda_{min}(Q^\star)$ and $\lambda_{\ell} = \lambda_{max}(Q^\star)$.
We note the spectral condition number $\kappa$ of a matrix $A$ the number: $\kappa(A) := \frac{\lambda_{max}(A)}{\lambda_{min}(A)}$
(it is equivalent to the Euclidean condition number if the matrix $A$ is symmetric).

The following result estimates the \emph{asymptotic} behavior of the conditioning number $\kappa(P^{-1}K)$ as
the barrier $\mu$ is reduced to $0$.
\begin{proposition}
\label{prop:boundskappa}
  Let $P^{-1} K$ be the preconditioned matrix defined in \eqref{eq:precond1}.
  If the solution of \eqref{eq:kkt} satisfies strict complementarity and the active submatrix $N_\mathcal{A}$
  is full row rank, the spectral condition number $\kappa_\mu := \kappa(P^{-1} K)$ converges to
  \begin{equation}
    \label{eq:limitcond}
  \lim_{\mu \to 0} \; \kappa_\mu = \dfrac{\max\{1, \lambda_{\ell}\}}{\min\{1, \lambda_{1}\}} \; .
  \end{equation}
  In particular, this implies that the condition number of the preconditioned matrix is bounded on the entire ``large" neighborhood of the central path  \eqref{eq:centralpath}.
\end{proposition}
\begin{proof}
This follows as a corollary of Proposition \ref{prop:limitprecond}, noting that in the limit,
$\text{spec}(P^{-1} K) = \{1\} \cup \text{spec}(Q^\star)$.
\end{proof}

\begin{remark}
We have proved that the spectral condition number of the preconditioned matrix $\kappa_\mu := \kappa(P^{-1}K)$ is bounded along the interior-point iterations,
without using the regularization procedure described in~\cite{Gondzio_2010}.
Moreover, under our assumptions the spectral condition number converges to a fixed value depending only on
$\lambda_{min}(Q^\star)$ and $\lambda_{max}(Q^\star)$.
This is important, as the error made by the CG iterates can be bounded above by (see \cite[Chapter 10]{golub2013matrix}):
\begin{equation}
  \| \boldx_k - \boldx_\star \|_{P^{-1/2} K P^{-1/2}} \leq
  2 \Big( \dfrac{\sqrt{\kappa_\mu} - 1}{\sqrt{\kappa_\mu} + 1} \Big)^k \|\boldx_0 - \boldx_\star \|_{P^{-1/2} K P^{-1/2}} \;.
\end{equation}
The significance of this result is that, if bound stays constant on the entire problem class, then the algorithm presented here is scalable (in the sense that the total effort is bounded above by the number of preconditioned conjugate gradient iterations times the effort per iteration).
As different from \cite{Gondzio_2010}, our method does not require the computation of an incomplete
Cholesky factorization to update the preconditioner at each IPM iteration, which would be intractable in our case.
For the algorithm proposed here it can shown to be quasilinear (limited by the effort of the FFT applications which can be $O(n \log^p n)$ for some small factor $p$) and the scalability is indeed demonstrated in \S \ref{sec:num-results}. Given that key matrix $\lambda_{min}(Q^\star)$ is a submatrix of an outer product of orthogonal vectors, and a very small matrix if the solution is sparse, it can be expected that the limiting condition number \eqref{eq:limitcond} will be $O(1)$, similar to \cite{Fountoulakis_2013}. This statement can be made precise in probability if the rows of $M_\perp$ are drawn at random \cite{candes2005l1} and is partially responsible for the success of compressed sensing in terms of recovery power.
\end{remark}

\begin{remark}\label{r:jacek}
Proposition~\ref{prop:boundskappa} establishes a slightly different result
than in \cite{Fountoulakis_2013}, which uses
an alternative method based on the \emph{restricted isoperimetry property} (RIP) of the matrix $\Mperp$ to derive
similar bounds on the condition number of the preconditioned matrix. As that reference leverages
the problem's structure, the bounds derived in \cite{Fountoulakis_2013} require more stringent assumptions,
but hold at all IPM iterations. The result can be summarized as follows.
It uses a formulation slightly different than ours
for the compressed sensing problem: For $F = \big[ \Mperp , -\Mperp \big]$, it investigates
the behavior of the condensed KKT system
\begin{equation}
    \big( \Theta^{-1} + F F^\top \big) \Delta x = \boldr \; ,
\end{equation}
with $\Theta^{-1} = S^{-1} Z \in \mathbb{R}^{2n \times 2n}$ and $\boldr \in \mathbb{R}^{2n}$ a given right-hand-side.
Let $K =  \Theta^{-1} + F F^\top $ and the preconditioner $P = \Theta^{-1} + \rho I$, defined for $\rho = \frac{n}{m}$.
As we approach convergence, most entries in $\Theta^{-1}$ become very large.
For a given $C \gg 1$, we note the number of small components $l := \#(\Theta_j^{-1} < C)$.
Assuming:
\begin{itemize}
    \item[(P1)] $\Mperp$ has rows close to orthonormal (there exists a small $\delta$  such that $\| \Mperp \Mperp^\top - I \|_2 \leq \delta$);
    \item[(P2)] $\Mperp$ satisfies a variant of the RIP: every $l$ columns of $\Mperp$ are almost orthogonal and have similar norms (for every matrix $B$ composed of arbitrary $l$ columns of $\Mperp$, there exists $\delta_l$ such that $\|\rho B^\top B - I \|_2 \leq \delta_l$).
\end{itemize}
Under (P1) and (P2), the cited reference proves in \cite[Lemma 1, p.14]{Fountoulakis_2013} the eigenvalues of the matrix $P^{-1} K$ are clustered around $1$:
for all $\lambda \in \text{spec}(P^{-1} K)$,
\begin{equation}
\label{eq:boundsfountoulakis}
 | \lambda - 1 | \leq \delta_l + \frac{1}{4}
 \dfrac{\big( 1 + \delta - \rho + 2 \sqrt{1 + \delta} \big)^2}{\rho \delta_l C} \; ,
\end{equation}
The parameter $C$ depends on $\mu$: as $\mu$ decreases to $0$, we can prove that $\lim_{\mu \to 0} C(\mu) = +\infty$.
Using the bounds \eqref{eq:boundsfountoulakis} in our case, and noting that we assume row orthogonality (implying $\delta=0$), we get that at the limit
$|\lambda - 1| \leq \delta_l$, leading to the following alternative
bounds on the preconditioned matrix:
\begin{equation}
  \lim_{\mu \to 0} \; \kappa_\mu = \dfrac{1 + \delta_l}{1 - \delta_l} \; ,
\end{equation}
a very similar result to this work.

Our result, however, is asymptotic in nature, it depends only on the conditioning
of the matrix $Q^\star$ --- a subset of the matrix $\Mperp^\top \Mperp$ --- whereas the constant
$\delta_l$ used in \cite{Fountoulakis_2013} depends on all the data in $\Mperp^\top \Mperp$ and it is intimately connected to RIP in (P2). RIP is a very suitable assumption when the rows of $M_\perp$ are drawn at random~\cite{candes2005l1}. On the other hand, when the missing data pattern is fixed and very structured, as is the case in our application, verifying RIP is NP-hard \cite{weed2017approximately}. Moreover, as our analysis shows, the good behavior is important only for the ``terminal" submatrix $Q^\star$, something far more likely to occur in practice. While not resulting in a materially better a priori bound than \cite{Fountoulakis_2013} for this class, the less restrictive assumptions in our work allow in principle to explain good behavior of preconditioned conjugate interior point approaches for other application settings, including $\ell_1$-regularized quadratic programming (through the substitution $M_\perp^T M_\perp \rightarrow Q_q$, the matrix of the quadratic program, provided that the corresponding limit \eqref{eq:limitcond} is small enough).
\end{remark}

\section{Applications}\label{sec:applications}
We demonstrate our approach on two datasets, namely a synthetic dataset and a real dataset that arises in material science. Below, we describe the details of experimental setting.
\subsection{Synthetic Dataset}\label{sec:synth} We generate synthetic datasets that are noisy and have missing values. The true signal $\boldsymbol{x}\in\mathbb{R}^{N_x\times N_y\times N_z}$ is generated from
\begin{multline}
        \boldsymbol{x}_{ijk} = \left(cos\left(\frac{2\pi}{N_{x}}\cdot i \right) + 2 sin\left(\frac{2\pi}{N_{x}}\cdot i \right)\right)\times \left(cos\left(\frac{2\pi}{N_{y}}\cdot 2j \right) + 2 sin\left(\frac{2\pi}{N_{y}}\cdot 2j \right)\right)\\
        \times \left(cos\left(\frac{2\pi}{N_{z}}\cdot 3k \right) + 2 sin\left(\frac{2\pi}{N_{z}}\cdot 3k \right)\right)\;\;\text{for}\;\; i\in[N_{x}], j\in[N_{y}], k\in[N_{z}].
\end{multline}
The noise $\widetilde{\bolde}\in\mathbb{R}^{N_x\times N_y\times N_z}$ has all entries \textit{iid} drawn from a uniform distribution between $0$ and $1$.
To mimic the removal of the information around the Bragg peaks of the unperturbed crystalline structure, we intentionally introduce missing values into our synthetic dataset, with approximately 15\% of the data assumed to be missing on average.
We demonstrate our approach on 3D-datasets of sizes up to $560 \times 560 \times 560$ (see table \ref{tab:cpu_gpu_comparison_artificial}).

\subsection{Recovering the 3D-$\Delta$PDF in crystals}
\label{sec:real_datasets}

The pair distribution function analysis of powder samples (powder PDF) of crystals is a common tool for investigating disordered structures.
The powder PDF is computed by the Fourier transform of the total X-ray or neutron powder diffraction pattern of a sample and provides a direct measure for the real interatomic distances in a material.
PDFs from single crystals (3D-PDF) may be calculated by the Fourier transform either of the total single crystal diffraction pattern (total 3D-PDF) or of the diffuse scattering alone (3D-$\Delta$PDF) \cite{welberry2022diffuse}.
When the PDF is calculated in three dimensions, it is possible to remove the Bragg peaks before performing the transform by using a technique known as ``punch-and-fill" \cite{kobas2005structural}, so that the resulting vector maps only
include those whose probabilities that differ from the average structure.
This allows the disorder to be directly visualized without extensive modeling and vastly simplifies interpretation.
The punch-and-fill algorithm involves the removal of Bragg peaks and ``fills" the data in the region of the Bragg peaks. Typically, the missing data at the Bragg peak locations are interpolated such that the missing data satisfies the 3D Laplace difference equations \cite{RH2022}.
This has the following advantages: a) it results in a banded linear system of equations (tridiagonal for 1D, pentadiagonal for 2D and so on) (hence one can leverage optimized sparse linear solvers), b) another major advantage is instead of solving a single large linear system, one can solve smaller local linear systems corresponding to each Bragg peak, an embarrassingly parallel operation.
For example for a 3D volume data with 125M pixels (500 $\times$ 500 $\times$ 500), instead of solving a large sparse linear system of size 250,000,000, one can solve 1000 sparse linear systems of size 10,000 in parallel.
However, the “punch-and-fill” algorithm (PFA) can generate unintentional ripples (or noise) in the resulting 3D-$\Delta$PDF because of sharp edges in the data interpolation. In order to mitigate this distortion, we propose to exploit the fact that the 3D-$\Delta$PDFs should be inherently sparse,  and use compressed sensing \S \ref{sec:formulation} as a regularizer to recover the missing information. Figure \ref{fig:VO2} demonstrates the results for a Molybdenum Vanadium Dioxide crystal; the size of the dataset used here is 300 $\times$ 400 $\times$ 400. Figure \ref{fig:VO2} shows a slice of the volume. The left panel shows the result of a ``Punch-and-fill" algorithm where the missing values are filled using a Laplace interpolation method (see \cite{RH2022} for details) and the right panel shows the result of recovering the sparse DFT using an $\ell_1$ penalty minimization. We note that the interpolation introduces ripples and other artifacts, where as the $\ell_1$ penalty minimization described in this paper results in a much cleaner 3D-$\Delta$PDF, demonstrating the recovery potential of compressed sensing for real data.

While we will detail computational performance in the next section; we emphasize that using the combined innovations described in this paper allowed us to improve our solution time for the real data problem from five days to less than five minutes over the last twenty four months. The comparison involved changing both the code base and the computational architecture: considering all the improvements, the CPU to GPU improvement is only about a factor of 25, see Table~\ref{tab:cpu_gpu_comparison}. Nevertheless, the leap forward is unmistakable and we are currently preparing to explore further the benefits of such a technology for our diffuse scattering application.
\begin{figure}
\includegraphics[width=\textwidth,trim={0 10in 0 10in}]{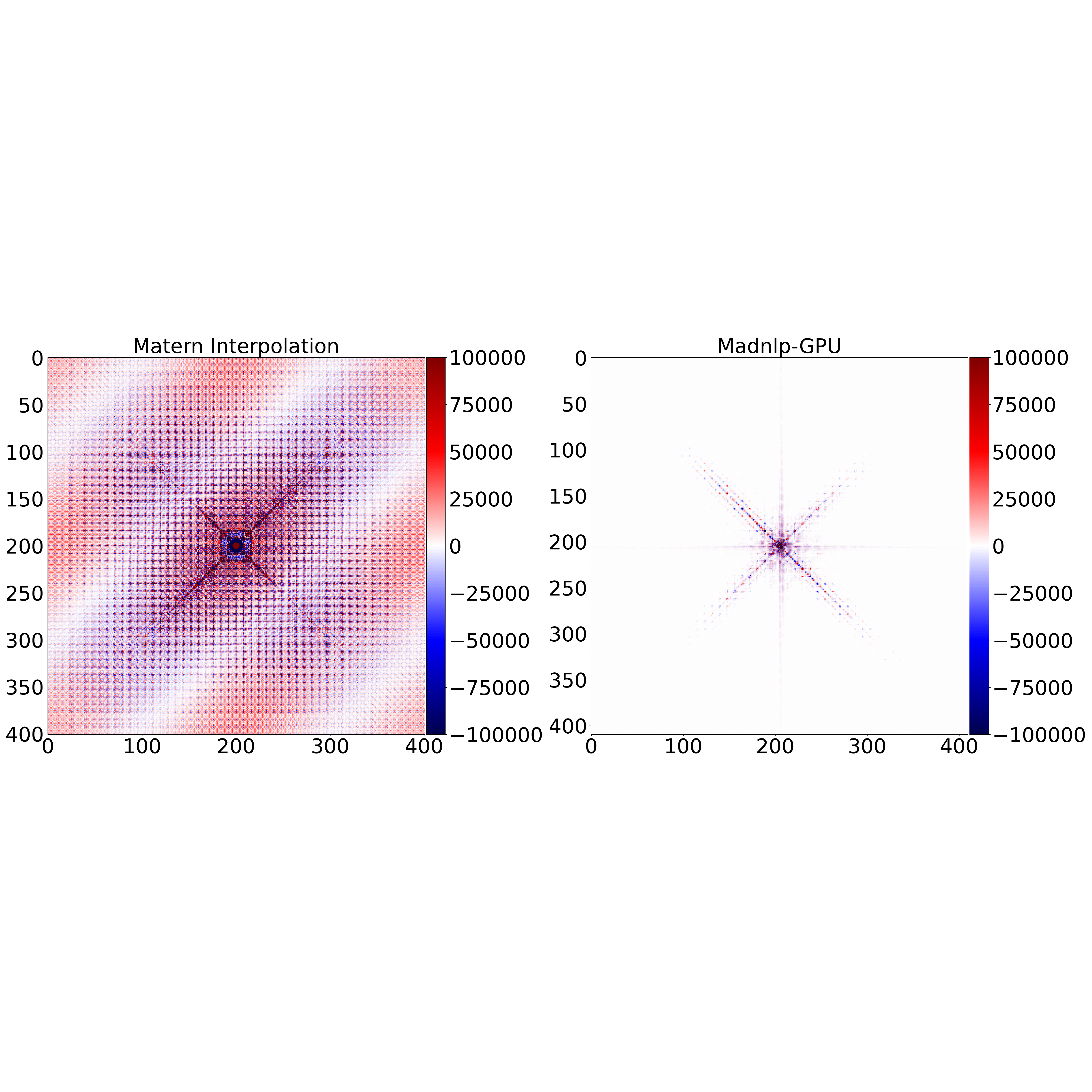}
\caption{\label{fig:VO2} Left panel shows the ordinary punch-and-fill, right panel shows the result obtained through $\ell_1$ constraint minimization approach.}
\end{figure}

\section{Implementation and numerical experiments}
\label{sec:num-results}

We developed our compressed sensing solver using Julia \cite{bezanson-edelman-karpinski-shah-2017}, version $1.11$.
The completed implementation is available the Julia package \texttt{CompressedSensingIPM.jl}\footnotemark.
\footnotetext{\url{https://github.com/exanauts/CompressedSensingIPM.jl}}

It relies on \texttt{MadNLP.jl} \cite{shin2023accelerating}, a generic primal-dual interior point method.
By default MadNLP solves the symmetric indefinite system~\eqref{eq:augkkt} at each IPM iteration.
We have developed a custom extension to solve instead the condensed system~\eqref{eq:condensedkkt}
with an iterative method, in a matrix-free fashion keeping all the matrices implicit.
To solve the KKT systems, we use the conjugate gradient method from the collection of Krylov methods \texttt{Krylov.jl} \cite{montoison-orban-2023}.
For forward and inverse Fourier transforms, our implementation employs the \texttt{FFTW} and \texttt{cuFFT} libraries, accessed through Julia’s interfaces \texttt{FFTW.jl} and \texttt{CUDA.jl}, respectively.
These libraries allow efficient computation on both CPUs and GPUs.

\texttt{MadNLP.jl} implements a vanilla interior-point method for NLP problems, and does not exploit the structure of the QP problem \eqref{eq:problem}.
The predictor-corrector method introduced by Mehrotra \cite{mehrotra-1992} would likely reduce the total number of interior-point iterations, but it would need to be adapted to limit the number of KKT systems~\eqref{eq:condensedkkt} to solve per iteration:
Unlike direct methods, we cannot reuse the factorization of the KKT matrix when solving the same KKT systems for two different right-hand-sides using a Krylov method.
A single-step predictor-corrector approach \cite{wright1997primal} is an interesting avenue to investigate.
This method could lead to speedups on both CPU and GPU architectures.
We leave this research direction for future work, along with the development of a potential specialized QP solver.

In our first experiment, we demonstrate the scalability of the solver on a problem with 104 million variables.
Using \texttt{MadNLP.jl}, the solver achieves convergence in 36 iterations, meeting a primal-dual residual threshold of $10^{-8}$.
Each iteration solves the KKT system with a stopping criterion based on an absolute tolerance of $10^{-12}$ for the preconditioned residual norm.

Matrix-free solvers are crucial in this context, as explicitly forming the KKT matrices is computationally prohibitive.
Dense linear solvers are not feasible due to the storage costs of dense DFT and IDFT matrices.
Despite the typical slow convergence of Krylov methods in interior-point solvers, our tailored preconditioner ensures that the number of Krylov iterations per KKT system remains below 105.

Figure~\ref{fig:krylov_iterations} illustrates the number of Krylov iterations required for each IPM iteration.

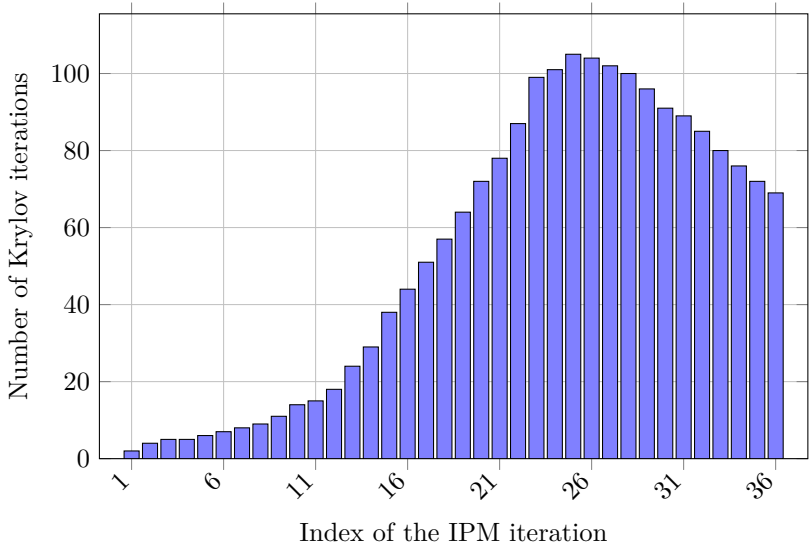
\begin{figure}[h!]
\centering
\begin{tikzpicture}
\begin{axis}[
    ybar,
    bar width=0.2cm,
    enlarge x limits=0.05,
    xlabel={Index of the IPM iteration},
    ylabel={Number of Krylov iterations},
    ymin=0,
    xtick={1, 6, 11, 16, 21, 26, 31, 36},
    xticklabel style={rotate=45, anchor=east},
    height=7.5cm,
    width=11cm,
    grid=both,
    major grid style={line width=0.2pt,draw=gray!50},
    minor grid style={line width=0.1pt,draw=gray!20}
]

\addplot[
    fill=blue!50!white
] coordinates {
    (1, 2) (2, 4) (3, 5) (4, 5) (5, 6) (6, 7)
    (7, 8) (8, 9) (9, 11) (10, 14) (11, 15) (12, 18)
    (13, 24) (14, 29) (15, 38) (16, 44) (17, 51) (18, 57)
    (19, 64) (20, 72) (21, 78) (22, 87) (23, 99) (24, 101)
    (25, 105) (26, 104) (27, 102) (28, 100) (29, 96) (30, 91)
    (31, 89) (32, 85) (33, 80) (34, 76) (35, 72) (36, 69)
};
\end{axis}
\end{tikzpicture}
\caption{Number of Krylov iterations per IPM iteration for a problem with 104 million variables.}
\label{fig:krylov_iterations}
\end{figure}

In our second experiment, we evaluate the total runtime for solving a problem with 104 million variables using actual X-ray scattering data on both CPU and GPU architectures (see \S\ref{sec:real_datasets} for more details).
Table~\ref{tab:cpu_gpu_comparison} presents the runtimes.

\begin{table}[ht!]
    \centering
    \begin{tabular}{ccc}
        \toprule
        \textbf{CPU} & \textbf{GPU} & \textbf{Speed-up} \\
        \midrule
        6,488 seconds & 274 seconds & $23.68\times$ \\
        \bottomrule
    \end{tabular}
    \vspace{0.3cm}
    \caption{Total runtimes on CPU and GPU architectures for X-ray diffuse scattering data with 104 million variables.}
    \label{tab:cpu_gpu_comparison}
\end{table}
These runtimes highlight that the GPU takes advantage of the parallelism inherent in the divide-and-conquer algorithms used for Fourier transforms.
The GPU also harnesses its computational power for vectorized linear algebra operations, such as vector updates, both in the Krylov solver and in the interior-point method.
These operations are accelerated through the use of \texttt{cuBLAS} and custom kernels generated via Julia's broadcasting capabilities.

In our third experiment, we evaluate the total runtime for solving artificial compressed sensing problems of various sizes on both CPU and GPU architectures (see \S\ref{sec:synth} for more details). Table~\ref{tab:cpu_gpu_comparison_artificial} presents the runtime results for each architecture.
The corresponding scalability plot, shown in~\Cref{fig:scalability_plot_cpu_gpu}, illustrates the performance scaling of both CPU and GPU across a range of problem sizes.
\begin{table}[ht!]
    \centering
    \begin{tabular}{crrrr}
        \toprule
        $N_x \times N_y \times N_z$ & Variables   & CPU & GPU  & Speed-up \\
        \midrule
        $8 \times 8 \times 8$       & 1,024       & 0.0031  & 0.63   & 0.005 \\
        $16 \times 16 \times 16$    & 8,192       & 0.013   & 0.72   & 0.018 \\
        $32 \times 32 \times 32$    & 65,536      & 0.10    & 0.68   & 0.147 \\
        $64 \times 64 \times 64$    & 524,288     & 1.29    & 0.71   & 1.82  \\
        $96 \times 96 \times 96$    & 1,769,472   & 4.49    & 0.79   & 5.64  \\
        $128 \times 128 \times 128$ & 4,194,304   & 10.34   & 1.07   & 9.67  \\
        $192 \times 192 \times 192$ & 14,155,776  & 41.78   & 3.22   & 12.98 \\
        $256 \times 256 \times 256$ & 33,554,432  & 101.38  & 5.88   & 17.24 \\
        $384 \times 384 \times 384$ & 113,246,208 & 367.78  & 20.66  & 17.80 \\
        $512 \times 512 \times 512$ & 268,435,456 & 929.26  & 104.55 & 8.89  \\
        $560 \times 560 \times 560$ & 351,232,000 & 1290.77 & 159.66 & 8.08  \\
        \bottomrule
    \end{tabular}
    \vspace{0.3cm}
    \caption{Performance results for an artificial 3D compressed sensing problem.}
    \label{tab:cpu_gpu_comparison_artificial}
\end{table}

\begin{figure}[ht]
    \centering
    \begin{tikzpicture}
    \begin{loglogaxis}[width=10cm, height=7.5cm,
                       xlabel={Number of variables},
                       ylabel={Time (seconds)},
                       legend pos=north west,
                       grid=both,
                       grid style={dotted, gray},
                       xtick={1e3,1e4,1e5,1e6,1e7,1e8,1e9},
                       ytick={1e-3,1e-2,1e-1,1e0,1e1,1e2,1e3},
                       major grid style={solid, gray},
                       minor tick num=0,
                       xmajorgrids,
                       ymajorgrids,
                       tick align=outside,
                       tick label style={font=\small},
                       legend style={font=\small}]

        \addplot[color=blue, mark=o, mark options={solid}, thick]
        table{
                1024 0.0031
                8192 0.013
                65536 0.10
                524288 1.29
                1769472 4.49
                4194304 10.34
                14155776 41.78
                33554432 101.38
                113246208 367.78
                268435456 929.26
                351232000 1290.77
            };
            \addlegendentry{CPU}

        \addplot[color=red, mark=square*, mark options={solid}, thick]
        table{
            1024       0.63
            8192       0.72
            65536      0.68
            524288     0.71
            1769472    0.79
            4194304    1.07
            14155776   3.22
            33554432   5.88
            113246208  20.66
            268435456  104.55
            351232000  159.66
        };
        \addlegendentry{GPU}
    \end{loglogaxis}
\end{tikzpicture}
\caption{Scalability of CPU and GPU performance across different problem sizes.}
\label{fig:scalability_plot_cpu_gpu}

\end{figure}
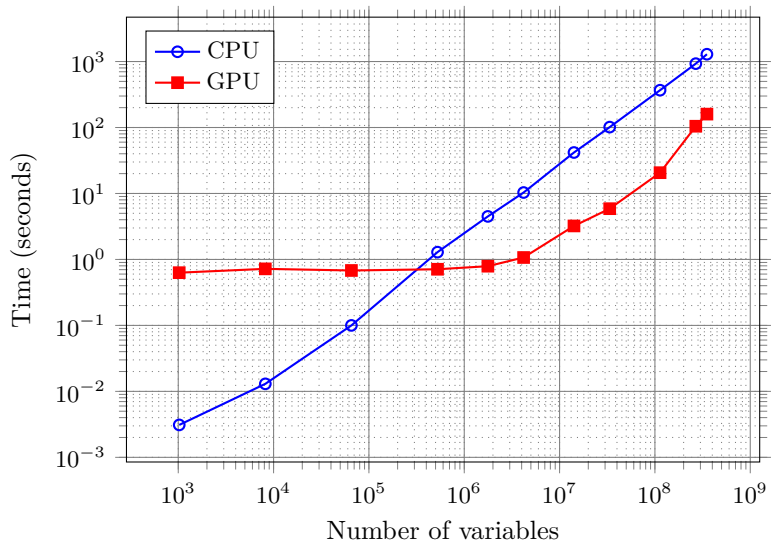

The results indicate that GPUs become more advantageous than CPUs for problems exceeding 500,000 variables, as shown in the scalability plot in Figure~\ref{fig:scalability_plot_cpu_gpu}.
To further investigate the performance benefits, we also conducted tests comparing the execution times of the FFT and IFFT on CPU and GPU architectures using the libraries \texttt{FFTW} and \texttt{cuFFT}.
Table~\ref{table:fft_ifft_times} compares the timings.
\begin{table}[ht!]
    \centering
    \begin{tabular}{|c|c|c|c|c|c|c|}
        \hline
        \multirow{2}{*}{$n$} & \multicolumn{3}{c|}{FFT} & \multicolumn{3}{c|}{IFFT}\\
        \cline{2-7}
          & \texttt{FFTW} & \texttt{cuFFT} & ratio
          & \texttt{FFTW} & \texttt{cuFFT} & ratio\\
        \hline
           $10$   & $1.47 \times 10^{-6}$ & $1.65 \times 10^{-5}$ & 0.09  & $1.47 \times 10^{-6}$ & $2.18 \times 10^{-5}$ & 0.07
        \\ $10^2$ & $3.17 \times 10^{-6}$ & $2.37 \times 10^{-5}$ & 0.13  & $3.39 \times 10^{-6}$ & $2.17 \times 10^{-5}$ & 0.16
        \\ $10^3$ & $1.55 \times 10^{-5}$ & $2.19 \times 10^{-5}$ & 0.71  & $1.18 \times 10^{-5}$ & $2.44 \times 10^{-5}$ & 0.48
        \\ $10^4$ & $8.29 \times 10^{-5}$ & $3.44 \times 10^{-5}$ & 2.41  & $8.59 \times 10^{-5}$ & $3.65 \times 10^{-5}$ & 2.35
        \\ $10^5$ & $1.30 \times 10^{-3}$ & $2.01 \times 10^{-4}$ & 6.26  & $9.59 \times 10^{-4}$ & $2.20 \times 10^{-4}$ & 4.36
        \\ $10^6$ & $1.64 \times 10^{-2}$ & $7.70 \times 10^{-4}$ & 21.34 & $1.60 \times 10^{-2}$ & $8.10 \times 10^{-4}$ & 19.70
        \\ $10^7$ & $1.59 \times 10^{-1}$ & $7.00 \times 10^{-3}$ & 22.69 & $1.67 \times 10^{-1}$ & $7.27 \times 10^{-3}$ & 22.97
        \\ $10^8$ & $2.19$    & $0.071$   & 30.66 & $2.36$    & $0.073$   & 32.48
        \\ $10^9$ & $27.17$   & $0.75$    & 35.90 & $29.03$   & $0.93$    & 31.27
        \\ \hline
    \end{tabular}
    \vspace{0.3cm}
    \caption{Execution time of Fourier transforms with \texttt{FFTW} and \texttt{cuFFT}.}
    \label{table:fft_ifft_times}
\end{table}
The tests were conducted on an NVIDIA GH200.
While our tests were conducted on NVIDIA GPUs, the code is designed to be portable.
By replacing \texttt{cuFFT} with \texttt{rocFFT}, support for AMD GPUs can be enabled through \texttt{AMDGPU.jl}.

\section{Conclusions and perspectives}\label{sec:conclusions}

Motivated by the problem of diffuse scattering in X-ray crystallography, this paper introduces a method for recovering sparse DFTs from noisy signals with missing values by using compressed sensing.
Instead of using a first-order method, we solve the resulting optimization problem
with an interior point method that has the potential to achieve much more robust convergence at the cost of solving more complex per-iteration problems.
The linear systems in the interior point algorithm can be solved efficiently by a preconditioned Krylov method: our numerical results have showed
that, even for very large problems, they need only up to around 100 Krylov iterations per Newton step. Our analysis framework supports these findings,
in accordance with that were also predicted by \cite{Fountoulakis_2013}.
However, while the resulting bounds on the preconditioned matrix we obtain are similar to \cite{Fountoulakis_2013}, the approach relies on different assumptions which are closer to the general analysis of optimization methods \cite{wright2001effects}.
As a result we believe it can be extended to other cases, such as certain quadratic programs, or matrices not derived from orthogonal operators to indicate that the condition number of the preconditioned matrix stays bounded there as well.

Our algorithmic findings were implemented in the openly available Julia package \texttt{CompressedSensingIPM.jl} that demonstrates the efficiency of solving large-scale compressed sensing problems formulated as a LASSO optimization problem on a GPU. The package was tested on problems with more than $10^8$ variables including one case with real X-ray crystallography data where compressed sensing shows a vast improvement over the classical punch-and-fill algorithm in terms of solution quality: the resulting algorithm runs in less than five minutes on a GPU.

We aim to improve and leverage these findings in multiple directions. First we aim to apply this to the entire database of diffuse scattering at Argonne, and prepare for the expected increase in the number of voxels (and thus problem size) through new sensors. Moreover, we expect that further improvements can be obtained by refining the sparsity model (since all atoms do show some width at the respective resolution) as well as using a windowing approach to control for aliasing artifacts.  Moreover, in the context of crystallographic X-ray imaging, periodic missing structures arising from regular atomic arrangements present a promising avenue for improving the preconditioning even further. Finally, we plan to extend these findings to other settings where the noise model results in nonlinear likelihoods and more complicated objective functions.

\section{Acknowledgment}

This research used resources of the Argonne Leadership Computing Facility, a U.S. Department of Energy (DOE) Office of Science user facility at Argonne National Laboratory and is based on research supported by the U.S. DOE Office of Science-Advanced Scientific Computing Research Program, under Contract No DE-AC02-06CH11357. We also thank Raymond Osborn, Stephan Rosenkranz, and Matt Krogstad for sharing the experimental data and providing valuable feedback on the results. We thank Charlotte Haley for her feedback on the spectral estimation problem in general.



\begin{thebibliography}{10}

\bibitem{ghannad-orban-saunders-2022}
{\sc D.~O. Alexandre~Ghannad and M.~A. Saunders}, {\em Linear systems arising in interior methods for convex optimization: a symmetric formulation with bounded condition number}, Optimization Methods and Software, 37 (2022), pp.~1344--1369, \url{https://doi.org/10.1080/10556788.2021.1965599}.

\bibitem{Becker_2011}
{\sc S.~Becker, J.~Bobin, and E.~J. Candès}, {\em {NESTA}: A fast and accurate first-order method for sparse recovery}, SIAM Journal on Imaging Sciences, 4 (2011), pp.~1--39, \url{https://doi.org/10.1137/090756855}.

\bibitem{bezanson-edelman-karpinski-shah-2017}
{\sc J.~Bezanson, A.~Edelman, S.~Karpinski, and V.~B. Shah}, {\em Julia: A fresh approach to numerical computing}, SIAM review, 59 (2017), pp.~65--98, \url{https://doi.org/10.1137/141000671}.

\bibitem{boyd2011distributed}
{\sc S.~Boyd, N.~Parikh, E.~Chu, B.~Peleato, J.~Eckstein, et~al.}, {\em Distributed optimization and statistical learning via the alternating direction method of multipliers}, Foundations and Trends{\textregistered} in Machine learning, 3 (2011), pp.~1--122.

\bibitem{byrd2006k}
{\sc R.~H. Byrd, J.~Nocedal, and R.~A. Waltz}, {\em {KNITRO}: An integrated package for nonlinear optimization}, Large-scale nonlinear optimization,  (2006), pp.~35--59.

\bibitem{candes2005l1}
{\sc E.~Candes, J.~Romberg, et~al.}, {\em l1-magic: Recovery of sparse signals via convex programming}, URL: www. acm. caltech. edu/l1magic/downloads/l1magic. pdf, 4 (2005), p.~16.

\bibitem{combettes2005signal}
{\sc P.~L. Combettes and V.~R. Wajs}, {\em Signal recovery by proximal forward-backward splitting}, Multiscale modeling \& simulation, 4 (2005), pp.~1168--1200.

\bibitem{Figueiredo_2007}
{\sc M.~A.~T. Figueiredo, R.~D. Nowak, and S.~J. Wright}, {\em Gradient projection for sparse reconstruction: Application to compressed sensing and other inverse problems}, IEEE Journal of Selected Topics in Signal Processing, 1 (2007), pp.~586--597, \url{https://doi.org/10.1109/jstsp.2007.910281}.

\bibitem{Fountoulakis_2013}
{\sc K.~Fountoulakis, J.~Gondzio, and P.~Zhlobich}, {\em Matrix-free interior point method for compressed sensing problems}, Mathematical Programming Computation, 6 (2013), pp.~1--31, \url{https://doi.org/10.1007/s12532-013-0063-6}.

\bibitem{golub2013matrix}
{\sc G.~H. Golub and C.~F. Van~Loan}, {\em Matrix computations}, JHU press, 2013.

\bibitem{Gondzio_2010}
{\sc J.~Gondzio}, {\em Matrix-free interior point method}, Computational Optimization and Applications, 51 (2010), pp.~457--480, \url{https://doi.org/10.1007/s10589-010-9361-3}.

\bibitem{gondzio2012interior}
{\sc J.~Gondzio}, {\em Interior point methods 25 years later}, European Journal of Operational Research, 218 (2012), pp.~587--601.

\bibitem{Hassanieh2012Nearly}
{\sc H.~Hassanieh, P.~Indyk, D.~Katabi, and E.~Price}, {\em Nearly optimal sparse {Fourier} transform}, in Proceedings of the forty-fourth annual ACM symposium on Theory of computing, STOC’12, ACM, May 2012, \url{https://doi.org/10.1145/2213977.2214029}.

\bibitem{Hassanieh2012Simple}
{\sc H.~Hassanieh, P.~Indyk, D.~Katabi, and E.~Price}, {\em Simple and practical algorithm for sparse {Fourier} transform}, in Proceedings of the Twenty-Third Annual ACM-SIAM Symposium on Discrete Algorithms, Society for Industrial and Applied Mathematics, Jan. 2012, \url{https://doi.org/10.1137/1.9781611973099.93}.

\bibitem{hestenes-stiefel-1952}
{\sc M.~R. Hestenes and E.~Stiefel}, {\em Methods of conjugate gradients for solving linear systems}, Journal of Research of the National Bureau of Standards, 49 (1952), pp.~409--436, \url{https://doi.org/10.6028/jres.049.044}.

\bibitem{Kim2007Interior}
{\sc S.-J. Kim, K.~Koh, M.~Lustig, S.~Boyd, and D.~Gorinevsky}, {\em An interior-point method for large-scale -regularized least squares}, IEEE Journal of Selected Topics in Signal Processing, 1 (2007), pp.~606--617, \url{https://doi.org/10.1109/jstsp.2007.910971}.

\bibitem{kobas2005structural}
{\sc M.~Kobas, T.~Weber, and W.~Steurer}, {\em Structural disorder in the decagonal al--co--ni. i. patterson analysis of diffuse x-ray scattering data}, Physical Review B—Condensed Matter and Materials Physics, 71 (2005), p.~224205.

\bibitem{mehrotra-1992}
{\sc S.~Mehrotra}, {\em On the implementation of a primal-dual interior point method}, SIAM Journal on optimization, 2 (1992), pp.~575--601, \url{https://doi.org/https://doi.org/10.1137/0802028}.

\bibitem{montoison-orban-2023}
{\sc A.~Montoison and D.~Orban}, {\em {Krylov.jl: A Julia basket of hand-picked Krylov methods}}, Journal of Open Source Software, 8 (2023), p.~5187, \url{https://doi.org/10.21105/joss.05187}.

\bibitem{montoison2023minares}
{\sc A.~Montoison, D.~Orban, and M.~A. Saunders}, {\em {MinAres}: An iterative solver for symmetric linear systems}, arXiv preprint arXiv:2310.01757,  (2023).

\bibitem{paige-saunders-1982}
{\sc C.~C. Paige and M.~A. Saunders}, {\em {LSQR}: An algorithm for sparse linear equations and sparse least squares}, ACM Trans. Math. Softw., 8 (1982), pp.~43--71, \url{https://doi.org/10.1145/355984.355989}.

\bibitem{RH2022}
{\sc V.~Rao, C.~L. Haley, and M.~Anitescu}, {\em {LaplaceInterpolation.jl: {A} {J}ulia package for fast interpolation on a grid}}, Journal of Open Source Software, 7 (2022), p.~3766, \url{https://doi.org/10.21105/joss.03766}.

\bibitem{saunders2002pdco}
{\sc M.~A. Saunders, B.~Kim, C.~Maes, S.~Akle, and M.~Zahr}, {\em {PDCO}: Primal-dual interior method for convex objectives}, Software available at http://www. stanford. edu/group/SOL/software/pdco. html,  (2002).

\bibitem{shin2023accelerating}
{\sc S.~Shin, M.~Anitescu, and F.~Pacaud}, {\em Accelerating optimal power flow with gpus: Simd abstraction of nonlinear programs and condensed-space interior-point methods}, Electric Power Systems Research, 236 (2024), p.~110651.

\bibitem{Smith_2012}
{\sc E.~Smith, J.~Gondzio, and J.~Hall}, {\em {GPU} Acceleration of the Matrix-Free Interior Point Method}, vol.~6, Springer Berlin Heidelberg, 2012, pp.~681--689, \url{https://doi.org/10.1007/978-3-642-31464-3_69}.

\bibitem{streiffer2015early}
{\sc S.~Streiffer, S.~Vogt, P.~Evans, et~al.}, {\em Early science at the upgraded advanced photon source}, Argonne National Laboratory, Tech. Rep,  (2015).

\bibitem{Takeshi2012Structure}
{\sc E.~Takeshi and S.~J. Billinge}, {\em Structure of Complex Materials}, Elsevier, 2012, pp.~1--25, \url{https://doi.org/10.1016/b978-0-08-097133-9.00001-0}.

\bibitem{Weber2012three}
{\sc T.~Weber and A.~Simonov}, {\em The three-dimensional pair distribution function analysis of disordered single crystals: basic concepts}, Zeitschrift für Kristallographie, 227 (2012), pp.~238--247, \url{https://doi.org/10.1524/zkri.2012.1504}.

\bibitem{weed2017approximately}
{\sc J.~Weed}, {\em Approximately certifying the restricted isometry property is hard}, IEEE Transactions on Information Theory, 64 (2017), pp.~5488--5497.

\bibitem{welberry2022diffuse}
{\sc T.~R. Welberry}, {\em Diffuse {X}-ray scattering and models of disorder}, vol.~31, Oxford University Press, 2022.

\bibitem{Welberry1994Interpretation}
{\sc T.~R. Welberry and B.~D. Butler}, {\em Interpretation of diffuse {X}-ray scattering via models of disorder}, Journal of Applied Crystallography, 27 (1994), pp.~205--231, \url{https://doi.org/10.1107/s0021889893011392}.

\bibitem{Wen_2010}
{\sc Z.~Wen, W.~Yin, D.~Goldfarb, and Y.~Zhang}, {\em A fast algorithm for sparse reconstruction based on shrinkage, subspace optimization, and continuation}, SIAM Journal on Scientific Computing, 32 (2010), pp.~1832--1857, \url{https://doi.org/10.1137/090747695}.

\bibitem{wright1997primal}
{\sc S.~J. Wright}, {\em Primal-dual interior-point methods}, SIAM, 1997, \url{https://doi.org/https://doi.org/10.1137/1.9781611971453}.

\bibitem{wright2001effects}
{\sc S.~J. Wright}, {\em Effects of finite-precision arithmetic on interior-point methods for nonlinear programming}, SIAM Journal on Optimization, 12 (2001), pp.~36--78.

\bibitem{W_chter_2005}
{\sc A.~Wächter and L.~T. Biegler}, {\em On the implementation of an interior-point filter line-search algorithm for large-scale nonlinear programming}, Mathematical Programming, 106 (2005), pp.~25--57, \url{https://doi.org/10.1007/s10107-004-0559-y}.

\end{thebibliography}
\end{document}